\documentclass[final,1p,times,sort&compress]{elsarticle}%
\usepackage{amsmath}
\usepackage{amssymb}
\usepackage{amscd}
\usepackage{amsfonts}
\usepackage{enumerate}
\usepackage{mathrsfs}
\usepackage{xy}
\usepackage{stmaryrd}
\usepackage{graphicx}
\usepackage{fancybox}
\usepackage{portland}
\usepackage{color}
\usepackage{multirow}
\usepackage{exscale}
\usepackage{enumitem,array}%
\usepackage{subfigure}
\usepackage{extarrows}
\usepackage{amsfonts}
\usepackage{slashbox}

\def\RR{{\mathbb{R}}}

\newtheorem{mytheo}{Theorem}[section]

\newtheorem{mydef}[mytheo]{Definition}

\newtheorem{prop}[mytheo]{Proposition}

\newtheorem{algo}[mytheo]{Algorithm}

\newtheorem{rem}[mytheo]{Remark}
\newcommand{\norm}[1]{\left\Vert#1\right\Vert}
\newcommand{\abs}[1]{\left\vert#1\right\vert}

\newcommand{\eps}{\varepsilon}

\newcounter{remark}

\newcounter{problem}

\makeatletter
\def\@upcite#1#2{\textsuperscript{[{#1\if@tempswa , #2\fi}]}}
\makeatother
\newenvironment{proof}{\vspace{1ex}
{\it Proof. }\hspace{0.3em}}{\vspace{1ex}} \journal{ }

\begin{document}
\begin{frontmatter}
\title{Geometric continuous-stage exponential energy-preserving integrators   for charged-particle dynamics in a magnetic field  from normal to strong regimes}

\author[1]{Ting Li}
\author[1]{Bin Wang\corref{cor1}}


\address[1]{School of Mathematical and Statistics, Xi'an Jiaotong University,
710049 Xi'an, China.}


\ead{litingmath@stu.xjtu.edu.cn, wangbinmaths@xjtu.edu.cn}\cortext[cor1]{Corresponding author.}


\begin{abstract}
This paper is concerned with  geometric exponential energy-preserving integrators  for solving charged-particle
dynamics in a magnetic field from normal to strong regimes. We firstly formulate the scheme of the methods for the system in a uniform magnetic field by using the idea of continuous-stage methods, and then discuss its energy-preserving property. Moreover, symmetric conditions and order conditions are analysed. Based on those conditions, we propose two practical symmetric continuous-stage  exponential energy-preserving integrators of order up to  four. Then we extend the obtained methods to the system in a nonuniform  magnetic field and derive their properties including the symmetry, convergence and energy conservation.
Numerical experiments demonstrate the efficiency of the proposed methods in comparison with some existing
schemes in the literature.
\end{abstract}

\end{frontmatter}
 \vskip0.5cm \pagestyle{myheadings}
\thispagestyle{plain}

\section{Introduction}\label{sec:intro}
 The dynamics of charged particles under the influence of external electromagnetic field  are of paramount importance
in plasma physics and they have important applications \cite{Arnold,Arnold97,brizard07fon,S4,Possanner18}. For example, these equations appear in Vlasov equations \cite{VP1,Zhao,VP3,S2,VP8,S5,S3} which are of fundamental importance in tokamak plasmas.
This article is concerned with the numerical solution of the following charged-particle dynamics (CPD)
(see \cite{Hairer2018,Lubich2020,WZ})
 \begin{equation}\label{charged-particle sts-cons}
\begin{array}[c]{ll}
\ddot{x}(t)=\dot{x}(t) \times  \frac{B(x(t))}{\epsilon} +F(x(t)), \quad
x(t_0)=x_0,\quad \dot{x}(t_0)=\dot{x}_0,\ \ t\in[t_{0},T],
\end{array}
\end{equation}where  $x(t)\in \RR^3$  denotes  the position of a particle, $B(x)$ is a given magnetic field, $F(x)$
is the negative gradient of the scalar potential $U(x)$,   and  $\epsilon\in(0,1]$ is a dimensionless parameter inversely proportional to the strength of the magnetic field. In this paper, we consider two regimes of
	$\epsilon$.  For the  \emph{strong regime} $0<\epsilon\ll 1$,  the solution of \eqref{charged-particle sts-cons} has the highly oscillatory behavior and this case arises from many important applications such as the magnetic fusion.  In contrast, the \emph{normal regime}, where the solution is not highly oscillatory and the system is a normal dynamic.
For these both regimes, the solution of \eqref{charged-particle
sts-cons} exactly conserves the following   energy
\begin{equation}\label{energy of cha}
H(x,v)=\frac{1}{2}\abs{v}^2+U(x),
\end{equation}
where  $v=\dot{x}$.

In the past few decades, many kinds of effective numerical integrators
for the system \eqref{charged-particle sts-cons} have been proposed. For the  normal regime, Boris method is a popular integrator which  was firstly developed in \cite{Boris} and further studied in \cite{Hairer2017,Qin2013}.  Concerning structure-preserving  schemes on this topic,  many different kinds of methods
have been derived. Volume-preserving integrators were constructed in
\cite{He2015} and symmetric multistep methods were studied in \cite{Lubich2017,wangwu2020}. The researchers in \cite{Hairer2018} studied variational integrators and the authors of \cite{Mei2017,Tao2016,Zhang2016} formulated symplectic or
K-symplectic integrators. Recently, energy-preserving methods  for charged-particle dynamics  in the normal regime were  proposed
in \cite{Brugnano2020,Brugnano,liting1,Ricketson}.

For the strong regime, i.e. $0<\epsilon\ll1$ in  \eqref{charged-particle
sts-cons}, some novel numerical methods have been developed and analysed recently.  Long time analysis of variational integrators  were presented in  \cite{Hairer2018,wangwu2020} and two filtered Boris algorithms were formulated in \cite{Lubich2020} under the maximal ordering scaling. Meanwhile,  different kinds of uniformly accurate schemes were derived in  \cite{VP1,Zhao,VP3,VP8}. Unfortunately, these effective methods do not conserve the energy \eqref{energy of cha} exactly.  In order to get energy-preserving (EP) methods for CPD, an exponential energy-preserving integrator was  recently developed in \cite{wang2021} for (\ref{charged-particle
sts-cons}) under a uniform  magnetic field $B$. Moreover recently,    first-order splitting energy-preserving methods  were researched in   \cite{WZ} with a rigorous error analysis and  high-order splitting EP methods  were considered in \cite{cui2021} but without convergence analysis.

On the other hand, “continuous-stage” methods have been received
much attention.  In \cite{E. Hairer}, Hairer  proposed
continuous-stage Runge-Kutta  method for Hamiltonian systems.
Following the approach of that paper,   a sufficient and necessary
energy-preserving condition of continuous-stage Runge-Kutta methods was given in
\cite{Miyatake}.  Continuous-stage Runge-Kutta-Nystr\"{o}m  methods for solving second-order
ordinary differential equations were discussed in \cite{Tang2018}.
 Recently, \cite{wbin2020} constructed  a continuous-stage modified
 Leap-frog scheme for high dimensional semi-linear Hamiltonian wave equations.
On the basis of these previous publications, it follows that continuous-stage methods can have exact energy conservation and good convergence.  Moreover, exponential-type methods have been shown to be competitive in the solving of highly oscillatory systems \cite{S6,Hochbruck2005,Hochbruck2010,Hochbruck2009}. Therefore, this paper is
devoted to exploring continuous-stage exponential energy-preserving integrators
for solving charged-particle dynamics \eqref{charged-particle
sts-cons}    in a magnetic field from normal to strong regimes.

The aim of the work is to propose and analyze a kind of continuous-stage exponential energy-preserving  integrators for solving \eqref{charged-particle sts-cons} with $\epsilon\in(0,1]$. To get the  energy-preserving property, we take advantage of continuous-stage methods and exponential integrators. With the energy-preserving conditions derived in the paper,  continuous-stage  exponential integrators with energy conservation can be formulated. To obtain high accuracy, symmetry and order conditions are derived and based on which, two symmetric continuous-stage exponential energy-preserving integrators of order up to  four are presented. Compared with the existing work, the main contributions of this paper involve in two aspects. We combine continuous-stage methods with exponential integrators to get continuous-stage exponential integrators with energy conservation  for   \eqref{charged-particle sts-cons} in a magnetic field with two different regimes. Moreover, symmetry is considered for this novel class of integrators  and based on which, higher-order EP methods can be constructed and analysed. The proposed schemes succeed in equipping the favorable continuous-stage exponential integrators with symmetry and exact energy conservation in long times.

The rest of the paper
is organised as follows. In
Section \ref{sec: methods}, we first formulate the scheme of  integrators for  \eqref{charged-particle
sts-cons} in a uniform magnetic field $B$ and then the energy-preserving conditions and  symmetric
conditions  are derived. In Section
\ref{sec: err ana}, we analyse the convergence of the  integrators.
Section \ref{sec:prac meth} constructs two practical symmetric
continuous-stage exponential energy-preserving  integrators of order up to  four.
 Two numerical experiments  are performed in Section \ref{sec:tests}  to show the efficiency of our methods  in comparison with the Boris method,  the averaged vector field (AVF) method and a splitting EP method of \cite{WZ}.
 In Section \ref{sec:new}, we extend the obtained methods as well as their properties to the  CPD \eqref{charged-particle
sts-cons} in a nonuniform  magnetic field
 and carry out two numerical tests to show their performance. The last section includes the conclusions of this paper.

\section{The integrators and structure-preserving properties}\label{sec:
methods}
From this section to Section \ref{sec:tests}, we shall formulate, analyse and test the novel   integrators for  \eqref{charged-particle
sts-cons} in a uniform magnetic field
 $B=(B_1,B_2,B_3)^{\intercal}$, where $B_i \in \RR$ for $i=1,2,3.$ According to the definition of the cross
 product, it is arrived that
$ \dot{x} \times  B =  \widetilde{B} \dot{x},$ and
$$\widetilde{B}=\left(
                   \begin{array}{ccc}
                     0 & B_3 & -B_2 \\
                     -B_3 & 0 & B_1 \\
                     B_2 & -B_1 & 0 \\
                   \end{array}
                 \right).
$$
Then the charged-particle dynamics \eqref{charged-particle sts-cons}
can be rewritten as
\begin{equation}\label{1equal}
\frac{d}{dt}\left(
                   \begin{array}{ccc}
                   x \\
                    v \\
                   \end{array}
                 \right)=\left(
                   \begin{array}{ccc}
                     0 &  I\\
                     0 & K \\
                   \end{array}
                 \right)
                 \left(
                   \begin{array}{ccc}
                   x \\
                    v \\
                   \end{array}
                 \right)+\left(
                   \begin{array}{ccc}
                   0\\
                    F(x) \\
                   \end{array}
                 \right),
\end{equation}
with $K=\frac{1}{\epsilon}\widetilde{B}$.
 In order to derive effective methods, applying variation-of-constants formula to \eqref{1equal} gives the following expression of the exact solution.

 \begin{mytheo} \label{VOC-SPE} The  exact solution of  the CPD  \eqref{charged-particle sts-cons} in a uniform magnetic field $B$ can be
expressed as
\begin{equation}\label{VOC}
\left\{\begin{aligned}
x(t_n+h)&=x(t_n)+h^2
\int_{0}^1(1-\tau) \varphi_1\big((1-\tau) hK \big) F\big(x(t_n+h\tau)\big)
d\tau
 +h\varphi_1( hK ) v(t_n),\\
v(t_n+h)&=\varphi_{0}( hK )v(t_n)+h
\int_{0}^1
 \varphi_{0}\big((1-\tau) hK \big) F\big(x(t_n+h\tau)\big)  d\tau,
\end{aligned}\right.
\end{equation}
for any stepsize $h\geq0$ and $t_n=nh$.
Here the $\varphi$-functions are defined by (see
\cite{Hochbruck2005,Hochbruck2010,Hochbruck2009})
\begin{equation}\label{phi}
 \varphi_0(z)=e^{z},\ \ \varphi_k(z)=\int_{0}^1
e^{(1-\sigma)z}\frac{\sigma^{k-1}}{(k-1)!}d\sigma, \ \
k=1,2,\ldots.
\end{equation}
\end{mytheo}

Based on  these preliminaries and the idea of continuous-stage methods, we define the following integrators.
\begin{mydef}
An $s$-degree continuous-stage  exponential integrator
 for solving the CPD
\eqref{charged-particle sts-cons}  in a uniform magnetic field $B$ is defined as
\begin{equation}\label{CSEEP}
\left\{\begin{aligned} &X_{\tau}=x_{n}+ hC_\tau( hK ) v_{n}+h^2 \int_{0}^{1}{A}_{\tau\sigma}( hK )F (X_\sigma)d\sigma,\ \ \ 0\leq\tau\leq1,\\
&x_{n+1}=x_{n}+ h\varphi_1( hK ) v_{n}+h^2
\int_{0}^{1}\bar{B}_\tau( hK )F (X_\tau)d\tau,\\
&v_{n+1}=\varphi_0( hK )v_{n}+h
\int_{0}^{1}B_{\tau}( hK )F
(X_\tau)d\tau,
\end{aligned}\right.
\end{equation}
where $h$ is the stepsize,  $ X_{\tau} $ is a polynomial of degree $ s $ with respect to $
\tau $ satisfying \begin{equation}\label{X01} X_{0}=x_{n},\ \ \
X_{1}=x_{n+1},\end{equation}
 $ C_{\tau}(hK) $, $ \bar{B}_{\tau}( hK ) $ and $
B_{\tau}( hK ) $ are polynomials of
degree $ s $ for $ \tau $ and depend on $ hK$, and $ A_{\tau\sigma}( hK ) $ is a
polynomial of degree $ s $ for $ \tau $, and $ s-1 $ for $ \sigma $
and depend on $  hK  $. The $
C_{\tau}( hK ) $ is assumed to satisfy
\begin{equation}\label{def C}
C_{c_{i}}( hK )=c_{i}\varphi_{1}(c_{i} hK ),
\end{equation}
where $ c_{i} $ with $ i=1,2, \ldots , s+1 $ are the fitting nodes, and one of them should be 1.
\end{mydef}

In this paper,   we choose $0= c_{1}\leq c_2 \leq\ldots \leq  c_{s+1}=1$ and then get
from \eqref{def C}  that
$$ C_{0}( hK )=C_{c_{1}}( hK )=0 \   \ \ \ \ \textmd{and} \ \ \ \ C_{1}( hK )=C_{c_{s+1}}( hK )=\varphi_{1}( hK ).$$
The function
$C_{\tau}( hK ) $  is considered as
\begin{equation*}
C_{\tau}( hK )=\sum\limits_{i=1}^{s+1}L_{i}(\tau)c_{i}\varphi_{1}(c_{i} hK ),
\end{equation*}
where $L_{i}(\tau)$ for $i=1,2,\ldots,s+1$ are Lagrange interpolation
functions   $
L_{i}(\tau)=\prod_{j=1,j\neq
i}^{s+1}\frac{\tau-c_{j}}{c_{i}-c_{j}}$ for $i=1,2,\ldots,s+1.$

\begin{rem}
It is noted that the nonlinear part of  (\ref{CSEEP}) is uniformly bounded for $\epsilon\in(0,1]$, and  the computational cost  per time step is uniform in $\epsilon\in(0,1]$ when a nonlinear iteration solver is applied.   In contrast, the other energy-preserving methods \cite{Brugnano2020,Brugnano,liting1,Ricketson} for solving CPD  \eqref{charged-particle sts-cons}  have the stiffness in their nonlinear equations. When some iteration solver such as fixed-point  or Newton's method is used, its convergence depends on $1/\epsilon$ for a given $h>0$ and the iteration converges slowly or may even not converge for  small $\epsilon$.

\end{rem}

In what follows, we study the structure-preserving properties of the integrator \eqref{CSEEP}. The first theorem is about energy-preserving and the second is for symmetry.
\begin{mytheo}
 The integrator \eqref{CSEEP} for solving the CPD \eqref{charged-particle sts-cons} in a uniform magnetic field $B$ is energy-preserving, i.e.,
  $$ H(x_{n+1},v_{n+1})=H(x_{n},v_{n}),\ \ \ \textmd{for}\ \ n=0,1,\ldots,$$
   if the coefficients satisfy
\begin{equation}\label{EPTJ}
\begin{aligned}
\varphi_{0}(- hK )B_{\tau}( hK )=C_{\tau}^{'}(- hK ),\ \ \
B_{\tau}(- hK )B_{\sigma}( hK )
=A_{\tau\sigma}^{'}( hK )+A_{\sigma\tau}^{'}(- hK ),\end{aligned}
\end{equation}
where $ A_{\tau\sigma}^{'}( hK )=\frac{\partial}{\partial\tau}A_{\tau\sigma}( hK ) $ \textmd{and} $ C_{\tau}^{'}( hK )=\frac{d}{d\tau}C_{\tau}( hK ).$
\end{mytheo}

\begin{proof}
According to the scheme of   \eqref{CSEEP} and $ H(x,v) $
given in \eqref{energy of cha}, we have
\begin{equation}
\begin{aligned}\nonumber
&H(x_{n+1},v_{n+1})-H(x_{n},v_{n})
=\frac{1}{2}v_{n+1}^{\intercal}v_{n+1}+U(x_{n+1})-\frac{1}{2}v_{n}^{\intercal}v_{n}-U(x_{n})\\
=&\frac{1}{2}\big(\varphi_0( hK )v_{n}+h\int_{0}^{1}B_{\tau}( hK )F (X_\tau)d\tau \big)^{\intercal} \big(\varphi_0( hK )v_{n}\\
  &+h\int_{0}^{1}B_{\tau}( hK )F (X_\tau)d\tau \big)+\int_{0}^{1}\big(\nabla U(X_{\tau})\big)^{\intercal}dX_{\tau}-\frac{1}{2}v_{n}^{\intercal}v_{n}.\\
\end{aligned}
\end{equation}Here we have used the result
$\int_{0}^{1}\big(\nabla U(X_{\tau})\big)^{\intercal}dX_{\tau}=\int_{0}^{1} dU(X_{\tau})=U(x_{n+1})-U(x_{n}),$
which is obtained by considering $ X_{0}=x_{n} $ and $ X_{1}=x_{n+1}
$.
Inserting $ F(x)=-\nabla U(x)$ and using the third formula of \eqref{CSEEP}, one obtains
\begin{equation}
\begin{aligned}\nonumber
&H(x_{n+1},v_{n+1})-H(x_{n},v_{n})
=\frac{1}{2}v_{n}^{\intercal}\big(\varphi_{0}( hK )\big)^{\intercal}\varphi_{0}( hK )v_{n}
+\frac{h}{2}v_{n}^{\intercal}\big(\varphi_{0}( hK )\big)^{\intercal}
\int_{0}^{1}B_{\tau}( hK )F(X_{\tau})d\tau\\
&+\frac{h}{2}\big(\int_{0}^{1}B_{\tau}( hK )F(X_{\tau})d\tau\big)^{\intercal}
\varphi_{0}( hK )v_{n}
+\frac{h^{2}}{2}\big(\int_{0}^{1}B_{\tau}( hK )F(X_{\tau})d\tau\big)^{\intercal}
\int_{0}^{1}B_{\tau}( hK )F(X_{\tau})d\tau\\
&-\int_{0}^{1}F^{\intercal}(X_{\tau})d\Big(x_{n}+ hC_\tau( hK ) v_{n}+h^2 \int_{0}^{1}{A}_{\tau\sigma}( hK )F (X_\sigma)d\sigma\Big)
-\frac{1}{2}v_{n}^{\intercal}v_{n}.
\end{aligned}
\end{equation}
Since $ \widetilde{B} $ is skew-symmetric, one obtains that
$\big(\varphi_{0}( hK )\big)^{\intercal}=\varphi_{0}(- hK )=e^{- hK }$, which yields $ \big(\varphi_{0}( hK )\big)^{\intercal}\varphi_{0}( hK )=I $. Thus, it is arrived that
\begin{equation*}
\begin{aligned}
&H(x_{n+1},v_{n+1})-H(x_{n},v_{n})
=hv_{n}^{\intercal}\varphi_{0}(- hK )\int_{0}^{1}B_{\tau}( hK )F(X_{\tau})d\tau
-h\int_{0}^{1}F^{\intercal}(X_{\tau})C_{\tau}^{'}( hK )v_{n}d\tau\\
&+\frac{h^{2}}{2}\int_{0}^{1}\int_{0}^{1}F^{\intercal}(X_{\tau})B_{\tau}(- hK )
B_{\sigma}( hK )F(X_{\sigma})d\sigma d\tau -h^{2}\int_{0}^{1}\int_{0}^{1}F^{\intercal}(X_{\tau})A_{\tau\sigma}^{'}( hK )F(X_{\sigma})d\sigma d\tau\\
&=\frac{h^{2}}{2}\int_{0}^{1}\int_{0}^{1}F^{\intercal}(X_{\tau})\big(B_{\tau}(- hK )
B_{\sigma}( hK )
-2A_{\tau\sigma}^{'}( hK )\big)F(X_{\sigma})d\sigma d\tau\\
&\ \ \ \ +hv^{\intercal}\int_{0}^{1}\big(\varphi_{0}(- hK )B_{\tau}( hK )
-C_{\tau}^{'}(- hK )\big)F(X_{\tau})d\tau.\\
\end{aligned}
\end{equation*}
From the first equation of \eqref{EPTJ}, we obtain
\begin{equation}
\begin{aligned}\nonumber
&H(x_{n+1},v_{n+1})-H(x_{n},v_{n})
=\frac{h^{2}}{2}\int_{0}^{1}\int_{0}^{1}F^{\intercal}(X_{\tau})\big(B_{\tau}(- hK )B_{\sigma}( hK )
-2A_{\tau\sigma}^{'}( hK )\big)F(X_{\sigma})d\sigma
d\tau.
\end{aligned}
\end{equation}
By letting $ \tau \leftrightarrow \sigma $, the above
formula becomes
\begin{equation}
\begin{aligned}\nonumber
&H(x_{n+1},v_{n+1})-H(x_{n},v_{n})\\
=&\frac{h^{2}}{2}\int_{0}^{1}\int_{0}^{1}F^{\intercal}(X_{\sigma})\big(B_{\sigma}(- hK )B_{\tau}( hK )
-2A_{\sigma
\tau}^{'}( hK )\big)F(X_{\tau})d\sigma
d\tau\\
=&\frac{h^{2}}{2}\int_{0}^{1}\int_{0}^{1}\Big(F^{\intercal}(X_{\sigma})\big(B_{\sigma}(- hK )B_{\tau}( hK )
-2A_{\sigma
\tau}^{'}( hK )\big)F(X_{\tau})\Big)^{\intercal}d\sigma
d\tau\\
=&\frac{h^{2}}{2}\int_{0}^{1}\int_{0}^{1}F^{\intercal}(X_{\tau})\big(B^{\intercal}_{\tau}( hK )B^{\intercal}_{\sigma}(- hK )
-2(A_{\sigma\tau}^{'}( hK ))^{\intercal}\big)F(X_{\sigma})d\sigma
d\tau\\
=&\frac{h^{2}}{2}\int_{0}^{1}\int_{0}^{1}F^{\intercal}(X_{\tau})\big(B_{\tau}(- hK )B_{\sigma}( hK )
-2A_{\sigma\tau}^{'}(- hK )\big)F(X_{\sigma})d\sigma
d\tau.
\end{aligned}
\end{equation}
Adding the above two   results yields
\begin{equation}
\begin{aligned}\nonumber
H(x_{n+1},v_{n+1})-H(x_{n},v_{n})
=\frac{h^{2}}{2}\int_{0}^{1}\int_{0}^{1}F^{\intercal}(X_{\tau})\big(B_{\tau}(- hK )B_{\sigma}( hK )
-A_{\tau\sigma}^{'}( hK )-A_{\sigma\tau}^{'}(- hK )\big)F(X_{\sigma})d\sigma
d\tau.
\end{aligned}
\end{equation}
By the second equation of \eqref{EPTJ}, we have $ H(x_{n+1},v_{n+1})-H(x_{n},v_{n})=0 $. The proof is completed.
 \hfill $\blacksquare$
\end{proof}

 The next theorem considers symmetric conditions of the new integrator \eqref{CSEEP}.
\begin{mytheo}If  the coefficients of the integrator \eqref{CSEEP}
satisfy
\begin{equation}\label{symm}
\begin{aligned}
&\varphi_{1}( hK )B_{\tau}(- hK )
-\bar{B}_{\tau}(- hK )=\bar{B}_{1-\tau}( hK ),\quad \quad\
 \varphi_{1}( hK )-C_{\tau}(- hK )\varphi_{0}( hK )
=C_{1-\tau}( hK ),\\
&\varphi_{1}( hK )B_{\sigma}(- hK )
-\bar{B}_{\sigma}(- hK )-C_{\tau}(- hK )
\varphi_{0}( hK )B_{\sigma}(- hK ) +A_{\tau\sigma}(- hK )=A_{1-\tau,1-\sigma}( hK ),
\end{aligned}
\end{equation}
the integrator   is symmetric, i.e., by exchanging $n+1\leftrightarrow n$ and $h\leftrightarrow -h$, the scheme \eqref{CSEEP} remains the same.
\end{mytheo}

\begin{proof}
By exchanging $ x_{n+1}\leftrightarrow x_{n} $, $
v_{n+1}\leftrightarrow v_{n} $ and replacing $ h $ by $ -h $ in
the scheme \eqref{CSEEP}, one obtains
\begin{equation}\label{symm1}
\left\{\begin{aligned}
&X_{\tau}^{*}=x_{n+1}-hC_{\tau}(- hK )v_{n+1}
+h^{2}\int_{0}^{1}A_{\tau\sigma}(- hK )F(X_{\sigma}^{*})d\sigma,\\
&x_{n}=x_{n+1}-h\varphi_{1}(- hK )v_{n+1}+h^{2}\int_{0}^{1}\bar{B}_{\tau}(- hK )F(X_{\tau}^{*})d\tau,\\
&v_{n}=\varphi_{0}(- hK )v_{n+1}-h\int_{0}^{1}B_{\tau}(- hK )F(X_{\tau}^{*})d\tau.
\end{aligned}\right.
\end{equation}
It follows from the third formula of \eqref{symm1} that
\begin{equation}\label{symm2}
v_{n+1}=\varphi_{0}( hK )v_{n}
+h\varphi_{0}( hK )\int_{0}^{1}B_{\tau}(- hK )F(X_{\tau}^{*})d\tau.
\end{equation}
Inserting the above result into the second formula of \eqref{symm1} leads to
\begin{equation}\label{symm3}
\begin{aligned}
x_{n+1}
&=x_{n}+h\varphi_{1}(- hK )v_{n+1}-h^{2}\int_{0}^{1}\bar{B}_{\tau}(- hK )F(X_{\tau}^{*})d\tau\\
&=x_{n}+h\varphi_{1}(- hK )\big( \varphi_{0}( hK )v_{n}
+h\varphi_{0}( hK )\int_{0}^{1}B_{\tau}(- hK )F(X_{\tau}^{*})d\tau
\big)
-h^{2}\int_{0}^{1}\bar{B}_{\tau}(- hK )F(X_{\tau}^{*})d\tau\\
&=x_{n}+h^{2}\int_{0}^{1}
\big( \varphi_{1}(- hK )\varphi_{0}( hK )B_{\tau}(- hK )
- \bar{B}_{\tau}(- hK )\big)F(X_{\tau}^{*})d\tau
+h\varphi_{1}(- hK )\varphi_{0}( hK )v_{n}.
\end{aligned}
\end{equation}
Keeping the definition of $ \varphi $-functions \eqref{phi} in mind, we obtain
$ \varphi_{1}(- hK )\varphi_{0}( hK )=\varphi_{1}( hK ). $
Then the formula \eqref{symm3} becomes
\begin{equation}\label{symm4}
\begin{aligned}
x_{n+1}
=x_{n}+h\varphi_{1}( hK )v_{n}+h^{2}\int_{0}^{1}
\big( \varphi_{1}( hK )B_{\tau}(- hK )
- \bar{B}_{\tau}(- hK )\big)F(X_{\tau}^{*})d\tau.
\end{aligned}
\end{equation}
Substituting \eqref{symm2} and \eqref{symm4} into the first formula of \eqref{symm1} implies
\begin{equation}
\begin{aligned}
X_{\tau}^{*}
=&x_{n}
+h^{2}\int_{0}^{1}\big(\varphi_{1}( hK )B_{\sigma}
(- hK )
-C_{\tau}(- hK )\varphi_{0}( hK )B_{\sigma}
(- hK )\\
&-\bar{B}_{\sigma}(- hK )+A_{\tau\sigma}(- hK ) \big)F(X_{\sigma}^{*})d\sigma
+h\big( \varphi_{1}( hK )-C_{\tau}(- hK )\varphi_{0}( hK ) \big)v_{n}.
\end{aligned}\nonumber
\end{equation}
Thus, we have
\begin{equation}\label{symm5}
\left\{\begin{aligned}
&X_{\tau}^{*}
=x_{n}+h\big( \varphi_{1}( hK )-C_{\tau}(- hK )\varphi_{0}( hK ) \big)v_{n}
+h^{2}\int_{0}^{1}\big(\varphi_{1}( hK )
B_{\sigma}(- hK )\\
&\qquad-\bar{B}_{\sigma}(- hK )
-C_{\tau}(- hK )\varphi_{0}( hK )
B_{\sigma}(- hK )+A_{\tau\sigma}(- hK )                                     \big)F(X_{\sigma}^{*})d\sigma,\\
&x_{n+1}=x_{n}+h^{2}\int_{0}^{1}
\big( \varphi_{1}( hK )B_{\tau}(- hK )
- \bar{B}_{\tau}(- hK )\big)F(X_{\tau}^{*})d\tau
+h\varphi_{1}( hK )v_{n},\\
&v_{n+1}=\varphi_{0}( hK )v_{n}+h\varphi_{0}( hK )
\int_{0}^{1}B_{\tau}(- hK )F(X_{\tau}^{*})d\tau.
\end{aligned}\right.
\end{equation}
Replacing all indices $ \tau $ and $ \sigma $ in \eqref{CSEEP} by $ 1-\tau $ and $ 1-\sigma $, respectively. Under the following conditions
\begin{equation}\label{symm6}
\begin{aligned}
&\varphi_{1}( hK )B_{\tau}(- hK )
-\bar{B}_{\tau}(- hK )=\bar{B}_{1-\tau}( hK ),\\
& \varphi_{1}( hK )-C_{\tau}(- hK )\varphi_{0}( hK )
=C_{1-\tau}( hK ), \qquad \varphi_{0}( hK )B_{\tau}(- hK )=B_{1-\tau}( hK ),\\
&\varphi_{1}( hK )B_{\sigma}(- hK )
-\bar{B}_{\sigma}(- hK )-C_{\tau}(- hK )
\varphi_{0}( hK )B_{\sigma}(- hK ) +A_{\tau\sigma}(- hK )=A_{1-\tau,1-\sigma}( hK ),
\end{aligned}
\end{equation}
the scheme \eqref{symm5} and \eqref{CSEEP} are the same. Therefore, the integrator \eqref{CSEEP} is symmetric.\\
It is worth noting that the conditions  \eqref{symm6}  can be simplified as \eqref{symm}.
The proof of this theorem is complete.
 \hfill  $\blacksquare$
\end{proof}

\section{Convergence}\label{sec: err ana}
In this section, we analyse the convergence of the integrator \eqref{CSEEP} and the following theorem states the corresponding result.



\begin{mytheo}\label{order condition}
It is assumed that \eqref{charged-particle sts-cons} has sufficiently smooth solutions, and $ F: \mathbb{R}^{n}\rightarrow \mathbb{R}$ is sufficient differentiable in a strip along the exact solution. Moreover, let $ F $ be locally Lipschitz-continuous, i.e., there exists $ L>0 $ such that
$\norm{F(u(t))-F(\tilde{u}(t))}  \leq L \norm{ u(t)- \tilde{u}(t) }$
for all $ t \in[t_{0}, T]$. Assume that the uniform bound of the coefficients of \eqref{CSEEP}  is $\widetilde{C}$.
Under the above conditions,  if the stepsize $h$ satisfies $ h \leq \sqrt{\frac{1}{2 \widetilde{C} L}}$
and the following $r$th-order conditions are fulfilled
\begin{equation}\label{orderPCSEEP}
\begin{aligned}
&\norm{\int_{0}^{1}B_{\tau}( hK )\frac{\tau^{j}}{j!}d\tau-\varphi_{j+1}( hK )} \leq \alpha_j h^{r-j} ,
\qquad \qquad\qquad\qquad \ \ \  \ j=0,1,\ldots,r-1,\\
&\norm{\int_{0}^{1}\bar{B}_{\tau}( hK )\frac{\tau^{j}}{j!}d\tau-\varphi_{j+2}( hK )}\leq \beta_j  h^{r-1-j},
\qquad \qquad\qquad\quad\ \  \ \  \ \   j=0,1,\ldots,r-2,\\
&\norm{\int_{0}^{1}\int_{0}^{1}A_{\tau \sigma}( hK )\frac{\sigma^{j}}{j!}d\tau d\sigma-\int_{0}^{1}\tau^{j+2}\varphi_{j+2}(\tau hK )d\tau }\leq \gamma_j h^{r-2-j},
\  \ j=0,1,\ldots,r-3, \\
\end{aligned}
\end{equation}
then the convergence of the   integrator \eqref{CSEEP} is given by
\begin{equation*}
\begin{aligned}\norm{x(t_{n})-x_n}\leq CT \exp\big(T(C+hC)\big)h^{m},\ \ \
\norm{v(t_{n})-v_n}\leq CT\exp\big(T(C+hC)\big)h^{m},
\end{aligned}
\end{equation*} where $ \norm{\cdot}  $ denotes the $\mathrm{L}^{\infty}$-norm, and
$C>0$ is a generic constant independent of $\epsilon$ or the time step or $n$  but depends on $\widetilde{C}, L, \alpha_j,\ \beta_j,\ \gamma_j$ and  $\norm{ \frac{d^r}{dt^r}F(x(t)) }$.
Here $ m=min(r,s+1)$ with the positive integer $ s $ given in \eqref{def C} and the positive integer $r$ is determined by \eqref{orderPCSEEP}.
\end{mytheo}

\begin{proof}
\textbf{(I)}
We first present the local errors  bounds of the method
\eqref{CSEEP}. Inserting the exact solution
\eqref{VOC} into the method \eqref{CSEEP}, we
have
\begin{equation}
\left\{\begin{aligned}\label{PCSEEP1}
&x(t_{n}+\tau h)=x(t_{n})+hC_{\tau}( hK )v(t_{n})+h^{2}\int_{0}^{1}
A_{\tau \sigma}( hK )
\hat{F}(t_{n}+\sigma h)d \sigma +\triangle_{\tau},\\
&x(t_{n+1})=x(t_{n})+h\varphi_{1}( hK )v(t_{n})+h^{2}\int_{0}^{1}
\bar{B}_{\tau}( hK )
\hat{F}(t_{n}+\tau h)d \tau+\rho_{n+1},\\
&v(t_{n+1})=\varphi_{0}( hK )v(t_{n})+h\int_{0}^{1}B_{\tau}( hK )
\hat{F}(t_{n}+\tau h)d\tau+\rho_{n+1}^{'},
\end{aligned}\right.
\end{equation}
where $ \triangle_{\tau} $, $ \rho_{n+1} $, $ \rho_{n+1}^{'} $ present the discrepancies of the method \eqref{CSEEP}, and $ \hat{F}(t)\equiv F(x(t))$.\\
It follows from the   variation-of-constants formula that
\begin{equation*}\label{PCSEEP2}
\begin{aligned}
x(t_{n}+\tau h)=x(t_{n})+\tau h\varphi_{1}(\tau hK )v(t_{n})+h^2\int_{0}^{\tau}(\tau-\sigma)\varphi_{1}\big((\tau-\sigma) hK \big)\hat{F}(t_{n}+h\sigma)d\sigma.\\
\end{aligned}
\end{equation*}
Combining  with the first formula in \eqref{PCSEEP1}, one has
\begin{equation*}
\begin{aligned}
&\triangle_{\tau}=
h\big(\tau\varphi_{1}(\tau hK )-C_{\tau}( hK )\big)v(t_{n})\\
+&
h^{2}\int_{0}^{\tau}(\tau-\sigma)\varphi_{1}\big((\tau-\sigma) hK \big)\hat{F}(t_{n}+h\sigma)d\sigma   -h^{2}\int_{0}^{1}A_{\tau\sigma}( hK )\hat{F}(t_{n}+\sigma h)d\sigma.
\end{aligned}
\end{equation*}
By the condition  \eqref{def C} and the results of Lagrange interpolation, it is arrived that
$h\tau \varphi_{1}(\tau hK )-hC_{\tau}( hK )=\mathcal{O}(h^{s+1}).$
By using Taylor series and the interesting properties of $ \varphi $-functions:
$$ \varphi_{j+1}(z)=-z^{-1}\Big(\frac{1}{j!}-\varphi_{j}(z)\Big),\
\int_{0}^{1}\frac{(1-\tau)\varphi_{1}\big((1-\tau)z\big)\tau^{j}}{j!}d\tau=\varphi_{j+2}(z),
$$
 we obtain
\begin{equation*}
\begin{aligned}
\triangle_{\tau}
&=\mathcal{O}(h^{s+1})+{\tau}^2 h^2\int_{0}^{1}(1-z)\varphi_{1}\big(\tau(1-z) hK \big)\hat{F}(t_{n}+h \tau z)dz
-h^{2}\int_{0}^{1}A_{\tau\sigma}( hK )\hat{F}(t_{n}+\sigma h)d\sigma\\
&=\mathcal{O}(h^{s+1})+\sum\limits_{j=0}^{r-3}h^{j+2}\bigg({\tau}^{j+2}\int_{0}^{1}(1-\sigma)\varphi_{1}
\big(\tau(1-\sigma) hK \big)\frac{{\sigma}^{j}}{j!}d\sigma
-\int_{0}^{1}A_{\tau\sigma}( hK )\frac{{\sigma}^{j}}{j!}d \sigma\bigg)\hat{F}^{j}(t_{n})\\
&=\mathcal{O}(h^{s+1})+\sum\limits_{j=0}^{r-3}h^{j+2}\big({\tau}^{j+2}\varphi_{j+2}(\tau hK )+\mathcal{O}(h^{r})
-\int_{0}^{1}A_{\tau\sigma}( hK )\frac{{\sigma}^{j}}{j!}d \sigma\big)\hat{F}^{j}(t_{n})\mathcal{O}(h^{r}),
\end{aligned}
\end{equation*}
where $ \hat{F}^{(j)}(t) $ denotes the $j$th order derivative of $ F(x(t)) $ with respect to $ t $.\\
In a similar way, one gets
\begin{equation*}
\begin{array}[c]{ll}
&\rho_{n+1}=\sum\limits_{j=0}^{r-2}h^{j+2}\big(\varphi_{j+2}( hK )
-\int_{0}^{1}\bar{B}_{\tau}( hK )\frac{{\tau}^{j}}{j!}d \tau\big)\hat{F}^{j}(t_{n})
+\mathcal{O}(h^{r+1}),\\
&{\rho}^{'}_{n+1}=\sum\limits_{j=0}^{r-1}h^{j+1}\big(\varphi_{j+1}( hK )
-\int_{0}^{1}B_{\tau}( hK )\frac{{\tau}^{j}}{j!}d \tau\big)\hat{F}^{j}(t_{n})
+\mathcal{O}(h^{r+1}).
\end{array}
\end{equation*}
In the light of \eqref{orderPCSEEP}, the following results are true
\begin{equation}\label{PCSEEP3}
\begin{array}[c]{ll}
\norm{\triangle_{\tau}}\leq Ch^{m}, \ \ \ \
\norm{\rho_{n+1}}\leq Ch^{r+1}, \ \ \ \ \ \norm{\rho_{n+1}^{'}}\leq Ch^{r+1}.
\end{array}
\end{equation}

\textbf{(II)} On the basis of the above analysis, in what follows we
  show the global error bounds of the method \eqref{CSEEP}. Let
\begin{equation*}
e_{n}^{x}=x(t_{n})-x_{n},\ \ \ e_{n}^{v}=v(t_{n})-v(t_{n}), \ \ \ e_{n}^{Y}=(e_{n}^{x}, e_{n}^{v})^{\intercal},
\ \ \ E_{\tau}=x(t_{n}+\tau h)-X_{\tau}.
\end{equation*}
Subtracting the formula \eqref{CSEEP} from \eqref{PCSEEP1} yields
\begin{equation}
\left\{\begin{aligned}\label{error}
&E_{\tau}=e_{n}^{x}+\tau h \varphi_1(\tau hK )e_{n}^{v}
+h^2\int_{0}^{1}A_{\tau\sigma}( hK )\big(F\big(x(t_{n}+\sigma h)\big)-F(X_{\sigma})\big)d\sigma+\Delta_{\tau}+\mathcal{O}(h^{s+1}),\\
&e_{n+1}^{x}=e_{n}^{x}+h \varphi_{1}( hK ) e_{n}^{v}
+h^{2}\int_{0}^{1}\bar{B}_{\tau}( hK )\big(F\big(x(t_{n}+\tau h)\big)-F(X_{\tau})\big)d\tau+\rho_{n+1},\\
&e_{n+1}^{v}=\varphi_{0}( hK ) e_{n}^{v}
+h\int_{0}^{1}B_{\tau}( hK )\big(F\big(x(t_{n}+\tau h)\big)-F(X_{\tau})\big)d\tau+\rho^{'}_{n+1},
\end{aligned}\right.
\end{equation}
where the initial conditions are $ e_{0}^{x}=0 $, $ e_{0}^{v}=0 $. It can be observed that here we replace $ C_{\tau}( hK )$
by $\tau\varphi_{1}(\tau hK )$, and this brings the $\mathcal{O}(h^{s+1})$ term in \eqref{error}.
We can express the last two equation of \eqref{error} as\\
\begin{equation}\label{equal1}
\begin{aligned}
e_{n+1}^{Y}&= \left(
                   \begin{array}{ccc}
                     I & h \varphi_1( hK )\\
                     0 & \varphi_0( hK ) \\
                   \end{array}
                 \right)
                 e_{n}^{Y}
                + h\int_{0}^{1}
                 \left(
                   \begin{array}{ccc}
                    h \bar{B}_{\tau}( hK ) & 0\\
                     0 & B_{\tau}( hK ) \\
                   \end{array}
                 \right)
                 \left(
                   \begin{array}{ccc}
                   F\big(x(t_{n}+\tau h)\big)-F(X_{\tau})\\
                     F\big(x(t_{n}+\tau h)\big)-F(X_{\tau}) \\
                   \end{array}
                 \right)d \tau \\
                 & \quad +\left(
                   \begin{array}{ccc}
                   \rho_{n+1} \\
                   \rho^{'}_{n+1} \\
                   \end{array}
                 \right).
 \end{aligned}
  \end{equation}
It is noted that
\begin{equation*}
\begin{aligned}
\left(
                   \begin{array}{ccc}
                     I & h \varphi_1( hK )\\
                     0 & \varphi_0( hK ) \\
                   \end{array}
                 \right)=\left(
                   \begin{array}{ccc}
                     I & 0\\
                     0 & \varphi_0( hK ) \\
                   \end{array}
                 \right)+h\left(
                   \begin{array}{ccc}
                     0 &  \varphi_1( hK )\\
                     0 & 0 \\
                   \end{array}
                 \right).
 \end{aligned}
  \end{equation*}
 Then, we have
\begin{equation*}
\begin{aligned}
\norm{\left(
                   \begin{array}{ccc}
                     I & h \varphi_1( hK )\\
                     0 & \varphi_0( hK ) \\
                   \end{array}
                 \right)}\leq \norm{\left(
                   \begin{array}{ccc}
                     I & 0\\
                     0 & \varphi_0( hK ) \\
                   \end{array}
                 \right)}+h\norm{\left(
                   \begin{array}{ccc}
                     0 &  \varphi_1( hK )\\
                     0 & 0 \\
                   \end{array}
                 \right)}
                 \leq 1+\widetilde{C}h.
 \end{aligned}
  \end{equation*}
Meanwhile, we  note that in the follows $\|\cdot\|_{c}$ denotes the maximum norm for continuous functions which is defined as
\begin{equation}\label{con2}
\norm{E}_{c}=\max \limits_{\tau\in[0,1]} \norm{E_{\tau}},
\end{equation}
for a continuous $\mathbb{R}^{M}$-valued function $E_{\tau}$ on [0,1].
It follows from the first formula of \eqref{error} and \eqref{con2} that
\begin{equation*}\label{con5}
\norm{E_{\tau}} \leq  \norm{e_{n }^{x}}+ \tau h \norm{e_{n }^{v}}+h^2 \widetilde{C} L\norm{E}_{c}+\norm{\Delta_{\tau}}+C h^{s+1},
\end{equation*}
which leads to
\begin{equation*}
\norm{E}_{c} \leq  \norm{e_{n }^{x}}+h \norm{e_{n }^{v}}+h^2 \widetilde{C} L \norm{E}_{c}+\norm{\Delta_{\tau}}_{c}+C h^{s+1}.
\end{equation*}
If the condition
$ h \leq \sqrt{\frac{1}{2 \widetilde{C} L}}$ is satisfied, one has

\begin{equation*}\label{con6}
\norm{E}_{c} \leq  2(\norm{e_{n }^{x}}+h \norm{e_{n }^{v}})+2\norm{\Delta_{\tau}}_{c}+2 C h^{s+1}.
\end{equation*}
Using the definition of the $\mathrm{L}^{\infty}$-norm, we obtain
$$ \norm{e_{n }^{x}}\leq \norm{e_{n }^{Y}},\ \  \norm{e_{n }^{v}}\leq \norm{e_{n }^{Y}}.$$
Moreover, considering the fact that $\norm{ F\big(x(t_{n}+\tau h)\big)-F(X_{\tau})}\leq L \norm{E_{\tau}}$,  one gets
\begin{equation}\label{equal2}
\begin{aligned}
\norm{ F\big(x(t_{n}+\tau h)\big)-F(X_{\tau})}&\leq L \norm{E_{\tau}}\leq L\norm{E}_{c}
\leq 2L(\norm{e_{n}^{x}}+h\norm{e_{n}^{v}})+2\norm{\Delta_{\tau}}_{c}+2Ch^{s+1}\\
&\leq 2L\norm{e_{n}^{Y}}+2Lh\norm{e_{n}^{Y}}+Ch^m.
\end{aligned}
  \end{equation}
By inserting \eqref{equal2} into \eqref{equal1} yields
\begin{equation*}
\begin{aligned}
\norm{e_{n+1}^{Y}}
\leq \norm{e_{n}^{Y}}+Ch\norm{e_{n}^{Y}}+Ch^2\norm{e_{n}^{Y}}
+Ch^{m+1}+Ch^{r+1}.
 \end{aligned}
  \end{equation*}
 Then it is arrived that
\begin{equation*}
\begin{aligned}
\norm{e_{n+1}^{Y}}\leq \norm{e_{n}^{Y}}+h(C+hC)\norm{e_{n}^{Y}}+Ch^{m+1}.
 \end{aligned}
  \end{equation*}
Using Gronwall inequality, it is easy to get
\begin{equation*}
\begin{aligned}
\norm{e_{n+1}^{Y}}\leq CT\exp\big(T(C+hC)\big)h^{m}.
 \end{aligned}
  \end{equation*}
Therefore, we obtain the following estimations
\begin{equation*}
\norm{e_{n}^{x}}\leq CT\exp\big(T(C+hC)\big)h^{m},
\ \ \ \norm{e_{n}^{v}}\leq CT\exp\big(T(C+hC)\big)h^{m}.
\end{equation*}
The conclusion of the theorem is confirmed.\hfill  $\blacksquare$
\end{proof}

\section{Construction of practical integrators}\label{sec:prac meth}
We are now ready to consider the construction of practical methods. In this section, we will propose  second-order and four-order symmetric  continuous-stage exponential energy-preserving  integrators  based on the energy-preserving conditions \eqref{EPTJ}, symmetric conditions \eqref{symm}  and order conditions \eqref{orderPCSEEP}.
\subsection{Second-order integrator}
Let us start with a one-degree method whose coefficients have the
following form
\begin{equation*}\label{one}
 \begin{aligned}  &A_{\tau\sigma}(hK)=\tau a( hK ),\ \ \bar{B}_{\tau}( hK )=\bar{b}(hK),\ \  B_{\tau}(hK)=b( hK ).
\end{aligned}
\end{equation*}
When $s=1$, by the definition \eqref{def C} of $
C_{\tau}( hK ) $, we have $$
C_{\tau}^{'}( hK )=\frac{c_{1}}{c_{1}-c_{2}}
\varphi_{1}(c_{1} hK )+\frac{c_{2}}{c_{2}-c_{1}}\varphi_{1}(c_{2} hK ).$$

Firstly, by the  energy-preserving conditions \eqref{EPTJ}, we
obtain
\begin{equation*}\label{s1ep}
\begin{aligned}[c]
&\varphi_{0}(- hK )b( hK )
=\frac{c_{1}}{c_{1}-c_{2}}\varphi_{1}(-c_{1} hK )+\frac{c_{2}}{c_{2}-c_{1}}\varphi_{1}(-c_{2} hK ),\ \ b(- hK )b( hK )=a( hK )+a(- hK ).
\end{aligned}
\end{equation*}
Solving the above formulas  yields
\begin{equation}\label{barb}
\begin{aligned}[c]
&b( hK )=\frac{1}{2c_{1}-1}
\bigg(c_{1}\varphi_{1}(c_{1} hK )-(1-c_{1})\varphi_{1}\big((1-c_{1}) hK \big)\bigg),\\
&a( hK )=\varphi_{2}\big((2c_{1}-1) hK \big) \ \ \textmd{or} \ \ \ a( hK )=\varphi_{2}\big(-(2c_{1}-1) hK \big).
\end{aligned}
\end{equation}
 Then from the symmetric conditions \eqref{symm}, it follows that
\begin{equation}\label{a}
\begin{aligned}[c]
&\bar{b}( hK )=\frac{1}{2c_{1}-1}
\bigg(c_{1}^{2}\varphi_{2}(c_{1} hK )-(1-c_{1})^{2}\varphi_{2}\big((1-c_{1}) hK \big)\bigg),\ a( hK )=\varphi_{2}\big(-(2c_{1}-1) hK \big).
\end{aligned}
\end{equation}
Letting $ c_{1}=0, c_{2}=1 $ in the formulae \eqref{barb} and \eqref{a}
gives
\begin{equation*}\label{one2}
\begin{array}
[c]{ll}
a( hK )=\varphi_{2}( hK ),\ \ \ \  b( hK )=\varphi_{1}( hK ),\ \ \ \
\bar{b}( hK )=\varphi_{2}( hK ).
\end{array}
\end{equation*}
By the above construction,  it is noted that the following result
holds
$$
A_{1\sigma}( hK )=\varphi_{2}( hK )=\bar{B}_{\sigma}( hK ),$$
which makes the requirement \eqref{X01} be true.\\
Meanwhile, it can be checked easily that the  coefficients $\bar{B}_{\tau}$ and $B_{\tau}$  satisfy the second-order conditions \eqref{orderPCSEEP} with $r=2$. 
Thus, we obtain a  continuous-stage symmetric exponential
energy-preserving integrator of order two with the
following coefficients
\begin{equation*}
 c_{1}=0, \ \ \ c_{2}=1, \ \ \
a( hK )=\varphi_{2}( hK ),\ \ \ b( hK )=\varphi_{1}( hK ),\ \ \
\bar{b}( hK )=\varphi_{2}( hK ).
\end{equation*}
We denote this method by M1-C.
\subsection{Fourth-order integrator}
We now consider two-degree  energy-preserving and symmetric
scheme with the following coefficients:
\begin{equation*}\label{two1}
\begin{aligned}
&A_{\tau\sigma}( hK )
=a_{11}( hK )\tau+a_{12}( hK )\tau\sigma+a_{21}( hK )\tau^{2}
+a_{22}( hK )\tau^{2}\sigma,\\
&B_{\tau}( hK )=b_{1}( hK )+b_{2}( hK )\tau,
 \ \ \quad \bar{B}_{\tau}( hK )=\bar{b}_{1}( hK )+\bar{b}_{2}( hK )\tau.
\end{aligned}
\end{equation*}

In the light of the definition \eqref{def C}  with $ s=2 $, we have
$$ C^{'}_{\tau}( hK )=\frac{2\tau-1}{c_{2}(c_{2}-1)}  c_{2}\varphi_{1}(c_{2} hK )\\
+\frac{2\tau-c_{2}}{1-c_{2}}\varphi_{1}( hK ).$$
According to the energy-preserving conditions \eqref{EPTJ} and the
third formula of symmetric conditions \eqref{symm6}, one arrives at
\begin{equation*}
\begin{aligned}
&\varphi_{0}(- hK )\big(b_{1}( hK )+b_{2}( hK )\tau\big)
=\frac{2\tau-1}{c_{2}(c_{2}-1)}c_{2}\varphi_{1}(-c_{2} hK )+\frac{2\tau-c_{2}}{1-c_{2}}\varphi_{1}(- hK ),\\
&\varphi_{0}(- hK )\big(b_{1}( hK )+b_{2}( hK )\tau\big)
=b_{1}(- hK )+(1-\tau)b_{2}(- hK ).
\end{aligned}
\end{equation*}
Letting $ c_{2}=\frac{1}{2} $ in the above conditions yields
\begin{equation}\label{bandbbar}
\begin{aligned}
b_{1}( hK )=-2\varphi_{1}(hK/2)+3\varphi_{1}( hK ),\ \ \
b_{2}( hK )=4\varphi_{1}(hK/2)-4\varphi_{1}( hK ).
\end{aligned}
\end{equation}
 Then from the second formula of  energy-preserving conditions \eqref{EPTJ}, we obtain
\begin{equation*}
\begin{aligned}
&b_{1}(- hK )b_{1}( hK )
+\sigma b_{1}(- hK )b_{2}( hK )
+\tau b_{1}( hK )b_{2}(- hK )
+\tau\sigma b_{2}(- hK )b_{2}( hK )\nonumber\\
=&\big(a_{11}( hK )+a_{11}(- hK )\big)+\sigma \big(a_{12}( hK )+2 a_{21}(- hK )\big)\nonumber
+\tau \big(a_{12}(- hK )
+2a_{21}( hK )\big)\\
&+2\tau\sigma \big(a_{22}( hK )+
a_{22}(- hK )\big), \nonumber
\end{aligned}
\end{equation*}
which leads to
\begin{equation}
\begin{aligned}\label{tt7}
&b_{1}(- hK )b_{1}( hK )
=a_{11}( hK )+a_{11}(- hK ),\ \
b_{1}( hK )b_{2}(- hK )
=a_{12}(- hK )+2a_{21}( hK ),\\
&b_{2}(- hK )b_{2}( hK )
=2a_{22}( hK )+2a_{22}(- hK ).
\end{aligned}
\end{equation}
Substituting the coefficients $
b_{1}( hK ) $ and $
b_{2}( hK ) $ \eqref{bandbbar} into
\eqref{tt7} gives
\begin{equation*}
\begin{aligned}
a_{11}( hK )=4\varphi_{2}(hK/2)-3\varphi_{2}( hK ),
\ \ \qquad & a_{12}( hK )=-6\varphi_{2}(hK/2)+4\varphi_{2}( hK ),\\
a_{21}( hK )=-5\varphi_{2}(hK/2)+6\varphi_{2}( hK ),
\qquad &
a_{22}( hK )=8\varphi_{2}(hK/2)-8\varphi_{2}( hK ).
\end{aligned}
\end{equation*}
It follows from the first formula of symmetric conditions \eqref{symm} that
\begin{equation*}
\begin{aligned}
&\varphi_{1}( hK )\big(-2\varphi_{1}(-hK/2)+3\varphi_{1}(- hK )
+4\tau\varphi_{1}(-hK/2)-4\tau\varphi_{1}(- hK )\big)\nonumber \\
=&\bar{b}_{1}(- hK )+\tau \bar{b}_{2}(- hK )+\bar{b}_{1}( hK )+(1-\tau)\bar{b}_{2}( hK ).
\end{aligned}
\end{equation*}
After doing some calculations, we arrive at
\begin{equation*}
\begin{aligned}
&\bar{b}_{1}( hK )=-\varphi_{2}(hK/2)+3\varphi_{2}( hK ),\
\ \ \
 \bar{b}_{2}( hK )=2\varphi_{2}(hK/2)-4\varphi_{2}( hK ).
\end{aligned}
\end{equation*}
It is can be checked that
$$ A_{1\sigma}( hK )=-\varphi_{2}(hK/2)+
3\varphi_{2}( hK )+\sigma
\Big(2\varphi_{2}(hK/2)-4\varphi_{2}( hK )\Big)=\bar{B}_{\sigma}( hK ),$$
which yields the condition   \eqref{X01}.

  It can be checked that the coefficients of  this integrator satisfy all the fourth-order conditions   \eqref{orderPCSEEP} with $r=4$ and symmetric conditions \eqref{symm}. Thus this continuous-stage exponential
energy-preserving integrator (denoted by M2-C) is symmetric  and of order four whose  coefficients are summarized as below
\begin{equation*}
\begin{aligned}
&c_{1}=0,\ \ \ c_{2}=\frac{1}{2}, \ \ \ c_{3}=1,\\
&b_{1}( hK )=-2\varphi_{1}(hK/2)+3\varphi_{1}( hK ),
\ \qquad b_{2}( hK )=4\varphi_{1}(hK/2)-4\varphi_{1}( hK ), \\
&\bar{b}_{1}( hK )=-\varphi_{2}(hK/2)+3\varphi_{2}( hK ),
\ \ \ \ \ \ \ \quad \bar{b}_{2}( hK )=2\varphi_{2}(hK/2)-4\varphi_{2}( hK ),\\
&a_{11}( hK )=4\varphi_{2}(hK/2)-3\varphi_{2}( hK ),
\ \ \qquad a_{12}( hK )=-6\varphi_{2}(hK/2)+4\varphi_{2}( hK ),\\
&a_{21}( hK )=-5\varphi_{2}(hK/2)+6\varphi_{2}( hK ),
\qquad a_{22}( hK )=8\varphi_{2}(hK/2)-8\varphi_{2}( hK ).
\end{aligned}
\end{equation*}

\section{Numerical tests}\label{sec:tests}
In this  section, we carry out two   numerical experiments to show
the efficiency of our   two integrators M1-C and M2-C. The methods for comparison  are
chosen as follows:
\begin{itemize}
  \item BORIS: the Boris method of order two presented in \cite{Boris};
  \item AVF: the Averaged Vector Field  method (energy-preserving method) of order two presented in \cite{McLachlan99};
   \item SEP: the  splitting energy-preserving method of order one presented in \cite{WZ}.
     \end{itemize}
For implicit methods, we consider
fixed-point iteration and set the
error tolerance as $ 10^{-16} $ and  the maximum number of
iteration as 100 in each iteration. In order to test the accuracy and the energy conservation, we consider the error: $error:=\frac{|x_{n}-x(t_n)|}{|x(t_n)|}+\frac{|v_{n}-v(t_n)|}{|v(t_n)|}$ and the error of energy $e_{H}:=\frac{|H(x_{n},v_n)-H(x_0,v_0)|}{|H(x_0,v_0)|}$ for all the methods.

\noindent\vskip3mm \noindent\textbf{Problem 1.} As the first
problem, we consider the charged-particle system  \eqref{charged-particle sts-cons} of \cite{Lubich2017} with a constant magnetic field and a additional factor $1/\epsilon$.
We choose   the potential $U(x)=\frac{1}{100\sqrt{x_{1}^2+x_{2}^2}}$, the constant magnetic field
$B=(0,0,1)^{\intercal}$ and the initial values $x(0)=(0,0.2,0.1)^{\intercal}, v(0)=(0.09, 0.05, 0.2)^{\intercal}.$

Firstly, we  solve the problem in $[0,10]$ with $h=1/2^k$ for $k=3,\ldots,7$.
The global errors are shown in Figure
 \ref{fig:problem11}   for different $\epsilon$. Then we solve the problem
with    $h=\frac{1}{100}$ on the interval $ [0,1000]$. The results of energy conservation  are
shown in  Figure \ref{fig:problem12}. Besides the energy, it is shown in \cite{Lubich2017} that this system also has the conservation of the momentum $$
M(x,v)=(v_{1}+A_{1}(x))x_{2}-(v_{2}+A_{2}(x))x_{1},$$
where the $ A(x)$ is given by $ A(x)=-\frac{1}{2}x\times B(x)$.
To show the behaviour of the considered methods in the conservation of this quantity, we integrate this problem on $ [0,1000]$ with    $h=\frac{1}{100}$ and present the relative momentum errors $e_{M}:=\frac{|M(x_{n},v_n)-M(x_0,v_0)|}{|M(x_0,v_0)|}$   in Figure \ref{fig:problem13}.

\begin{figure}[t!]
\centering\tabcolsep=0.4mm
\begin{tabular}
[c]{ccc}%
\includegraphics[width=4.6cm,height=4.6cm]{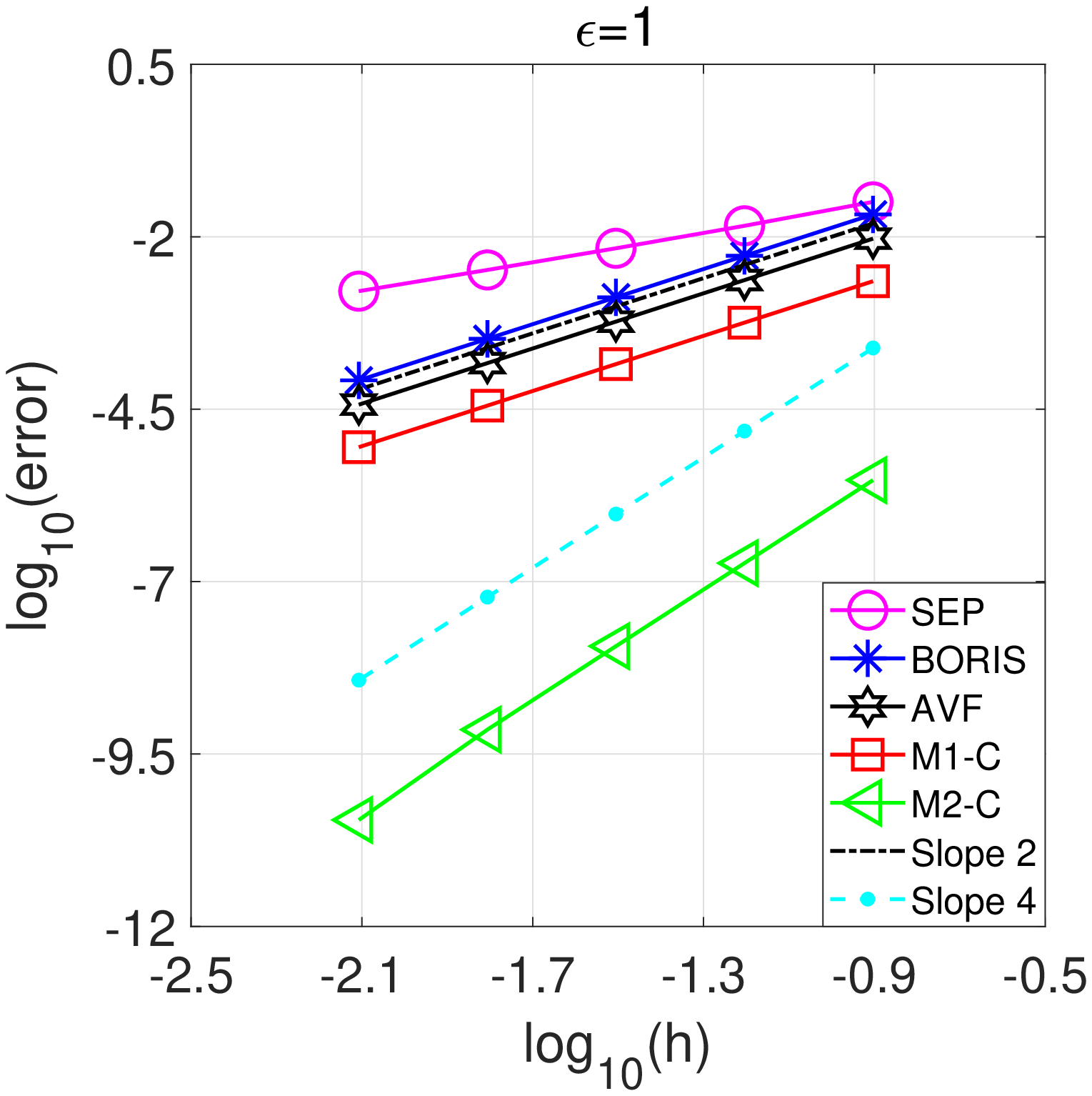} & \includegraphics[width=4.6cm,height=4.6cm]{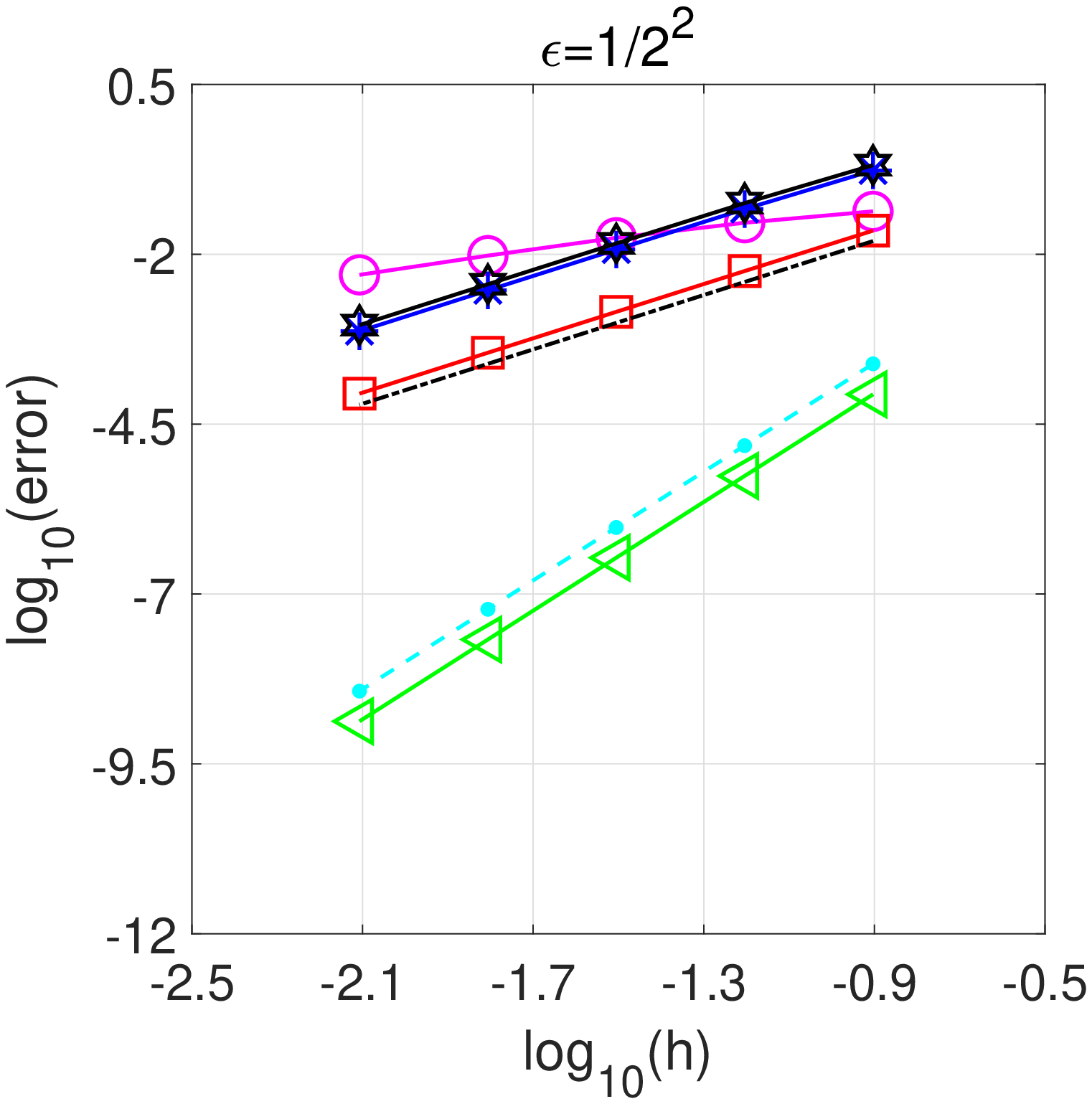}&\includegraphics[width=4.6cm,height=4.6cm]{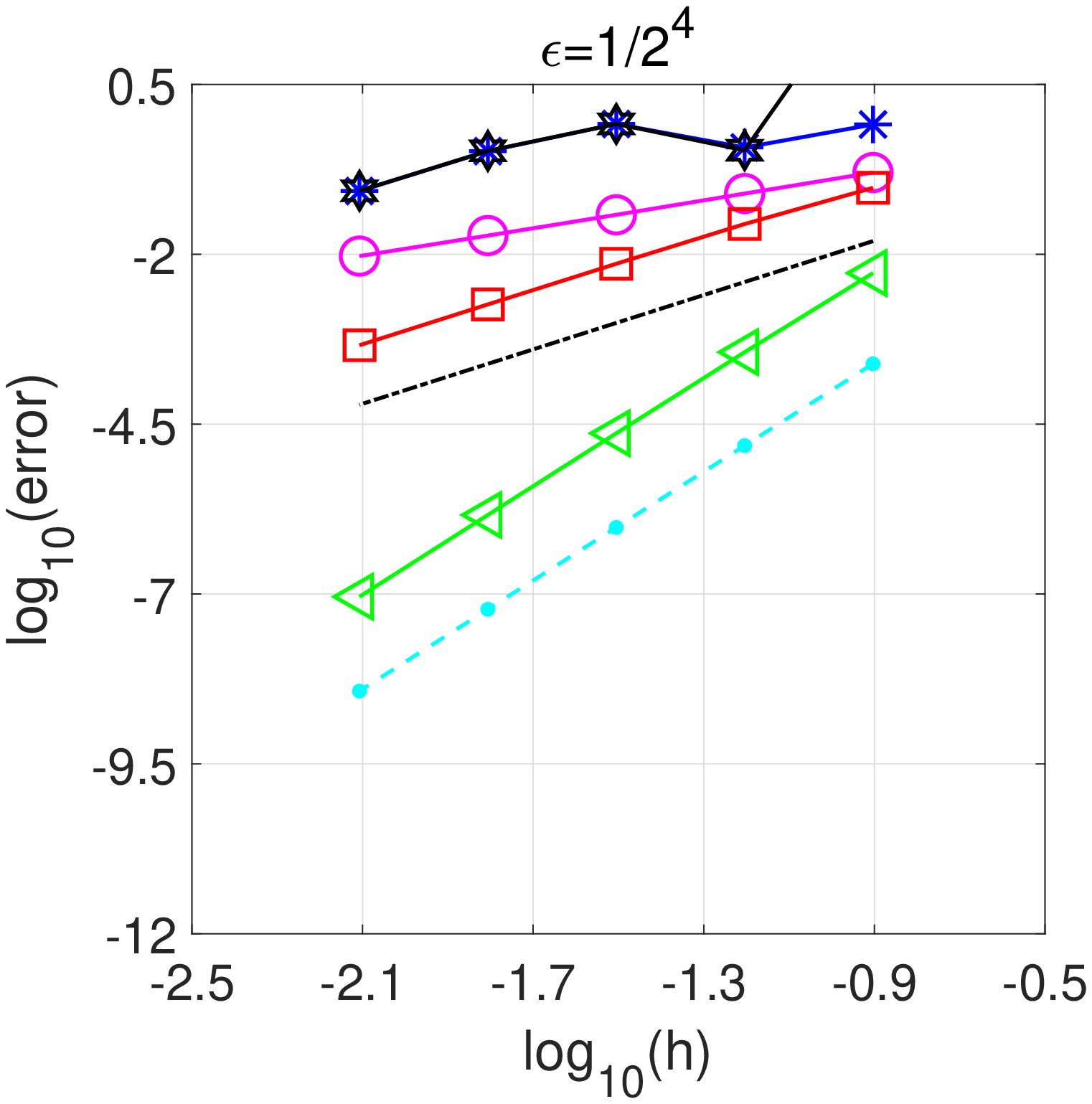}\\
{\small (i)} & {\small (ii)} & {\small (iii)}\\
\end{tabular}
\caption{Problem 1. The global errors $error:=\frac{|x_{n}-x(t_n)|}{|x(t_n)|}+\frac{|v_{n}-v(t_n)|}{|v(t_n)|}$ with $t=10$ and $h=1/2^{k}$ for $k=3,...,7$ under different $\epsilon$. }
\label{fig:problem11}
\end{figure}

\begin{figure}[t!]
\centering\tabcolsep=0.4mm
\begin{tabular}
[c]{ccc}%
\includegraphics[width=4.6cm,height=4.6cm]{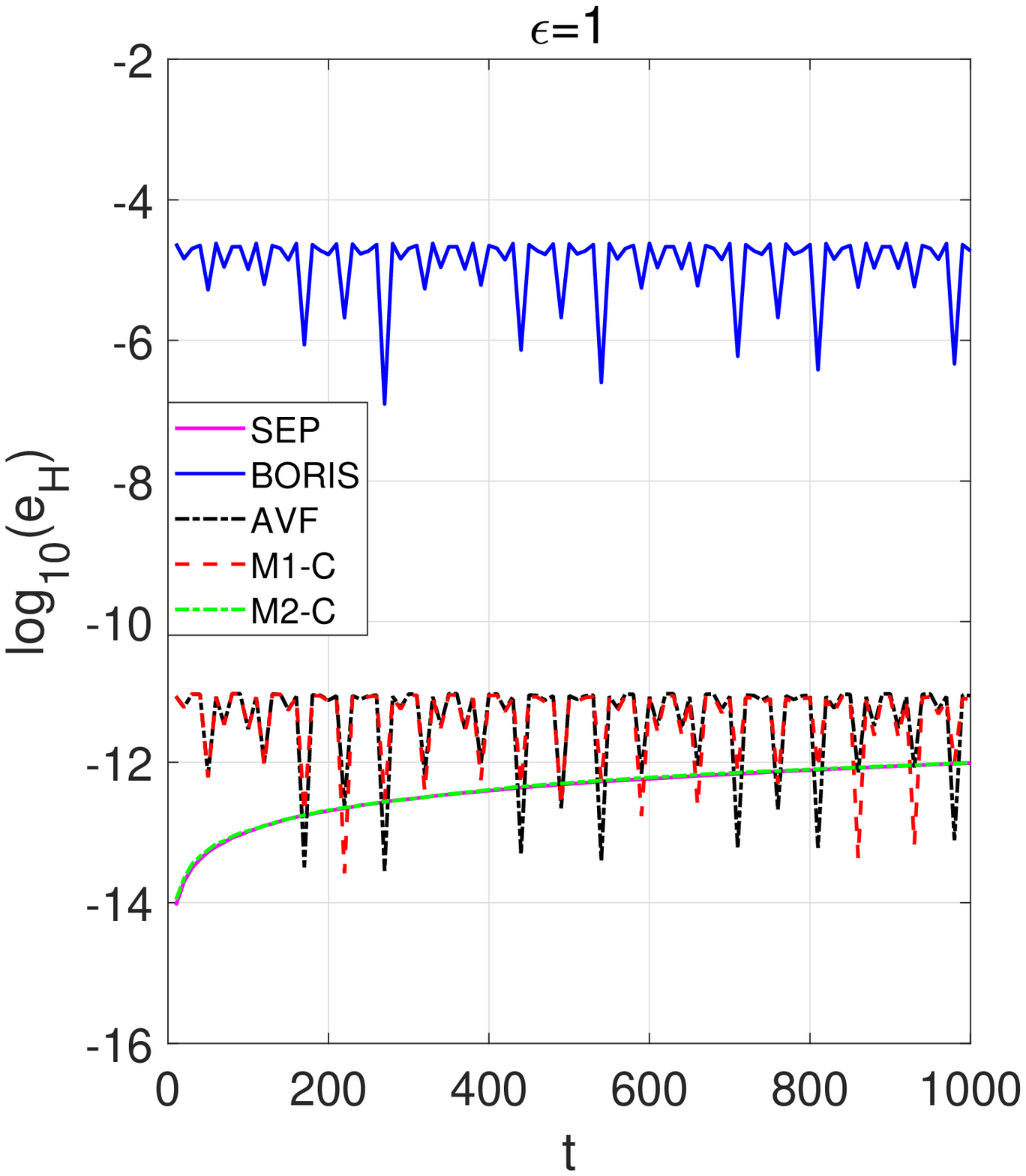} & \includegraphics[width=4.6cm,height=4.6cm]{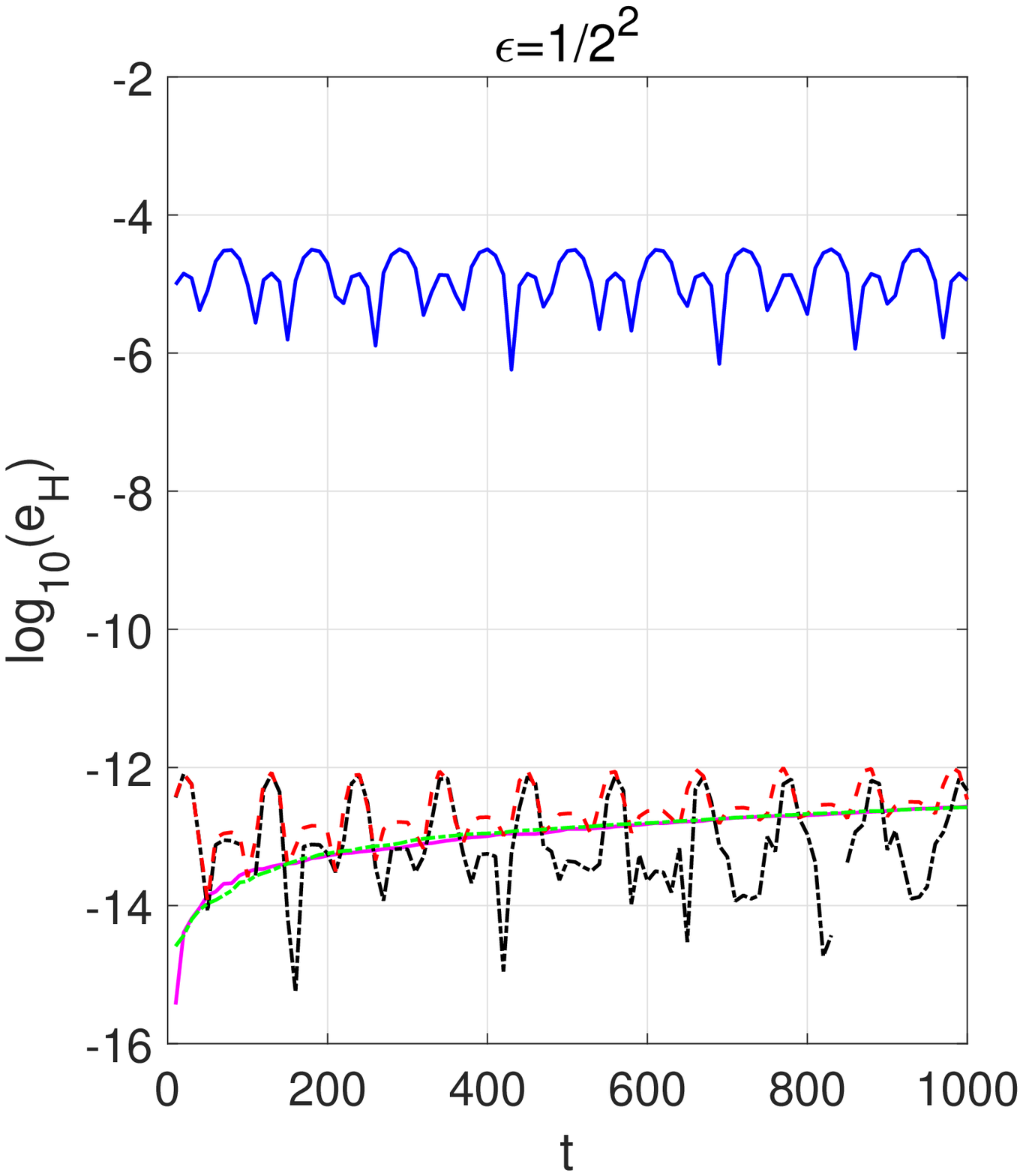} & \includegraphics[width=4.6cm,height=4.6cm]{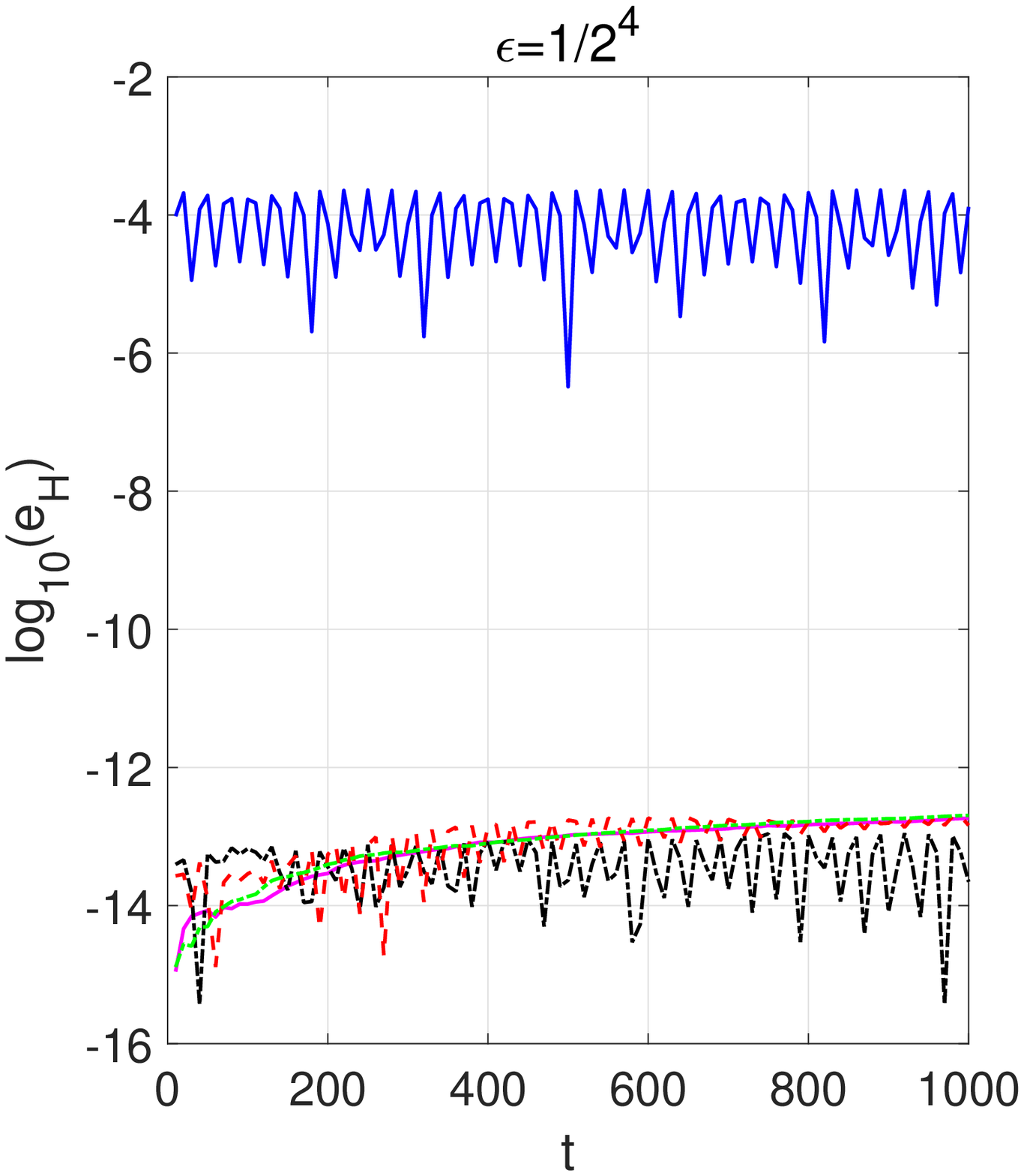}\\
{\small (i)} & {\small (ii)} & {\small (iii)}\\
\end{tabular}
\caption{Problem 1. Evolution of the energy error $e_{H}:=\frac{|H(x_{n},v_n)-H(x_0,v_0)|}{|H(x_0,v_0)|}$ as function of time  $t=nh$. }
\label{fig:problem12}
\end{figure}

\begin{figure}[h!]
\centering\tabcolsep=0.4mm
\begin{tabular}
[c]{ccc}%
\includegraphics[width=4.6cm,height=4.6cm]{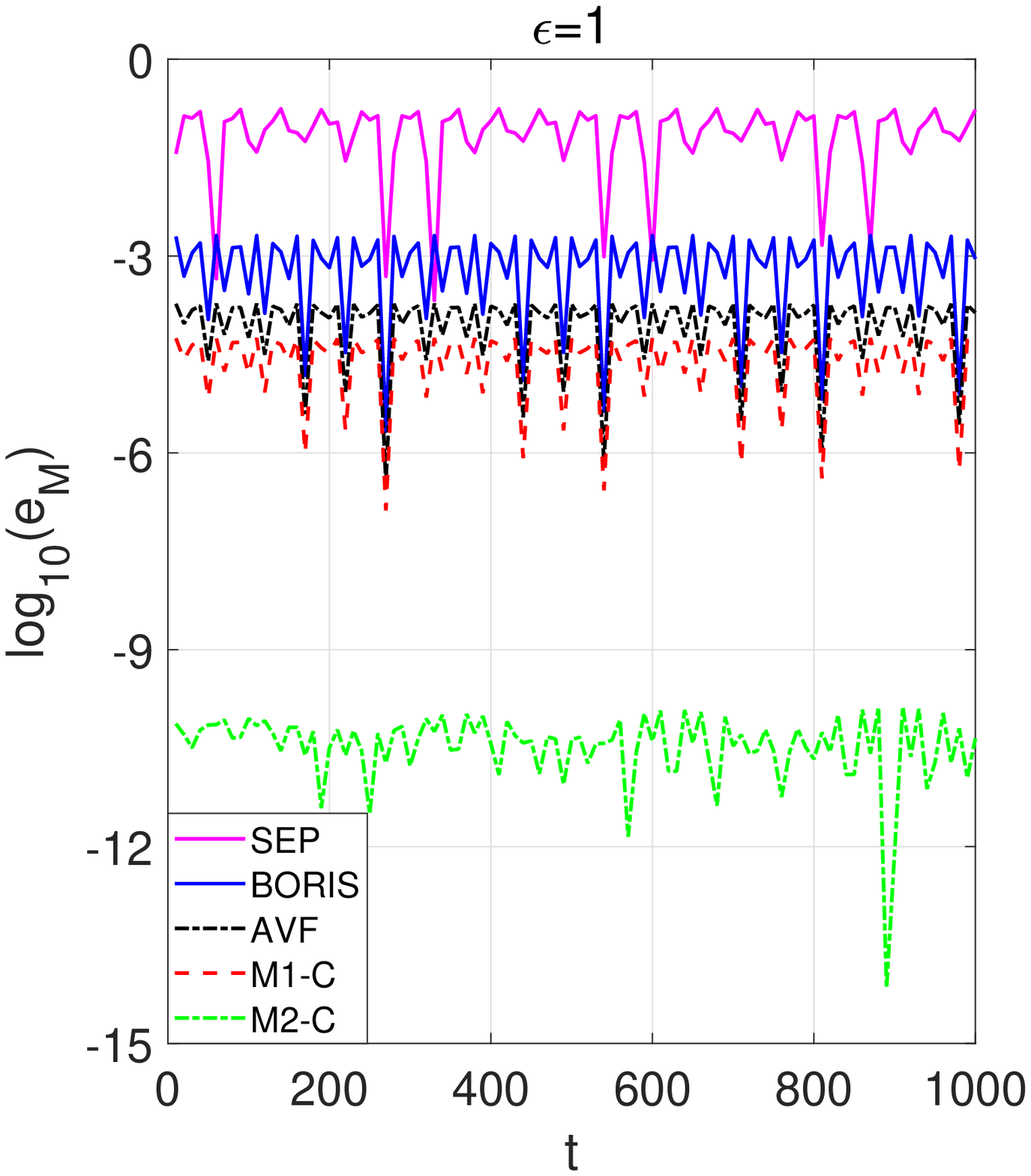} & \includegraphics[width=4.6cm,height=4.6cm]{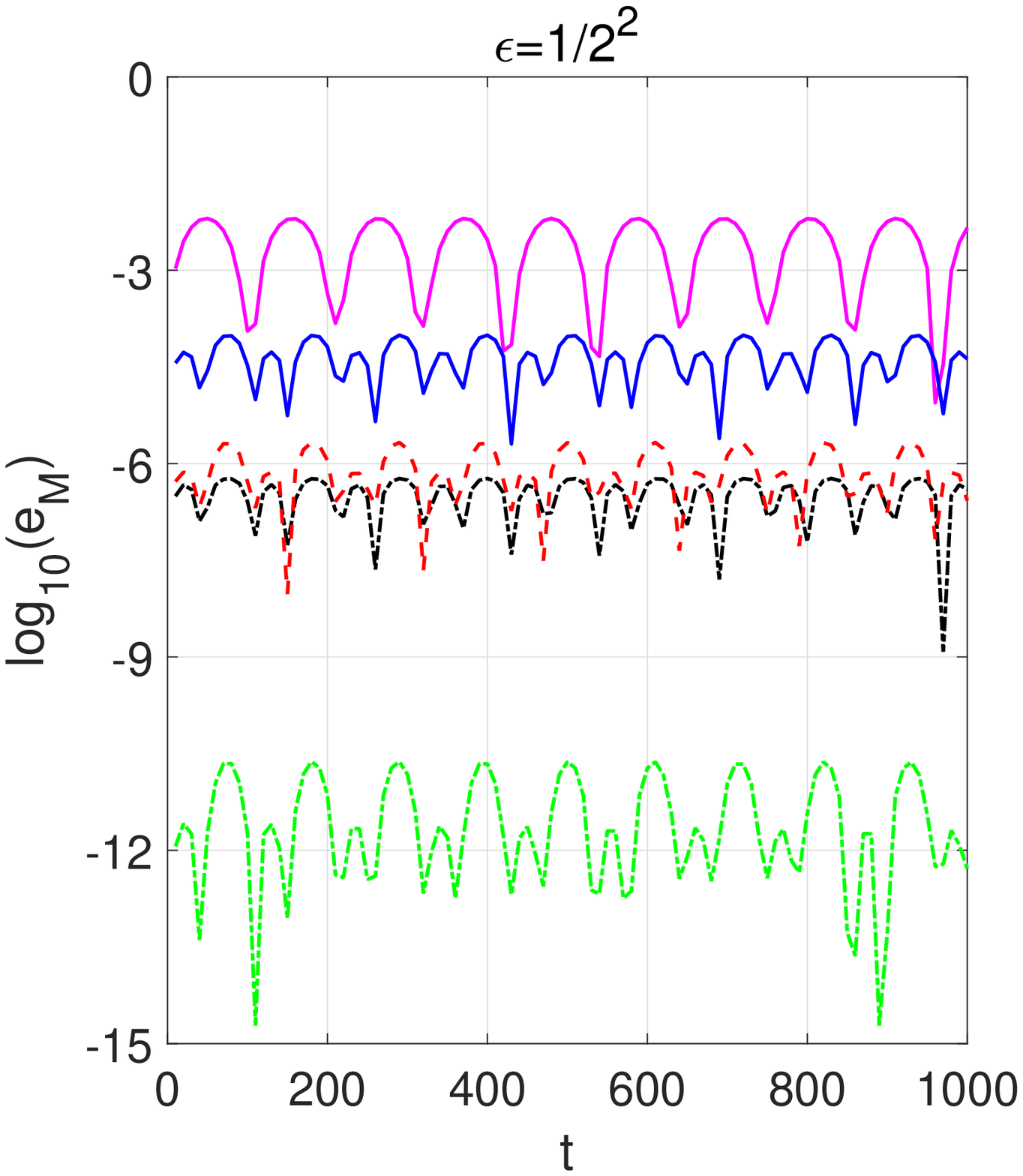}&\includegraphics[width=4.6cm,height=4.6cm]{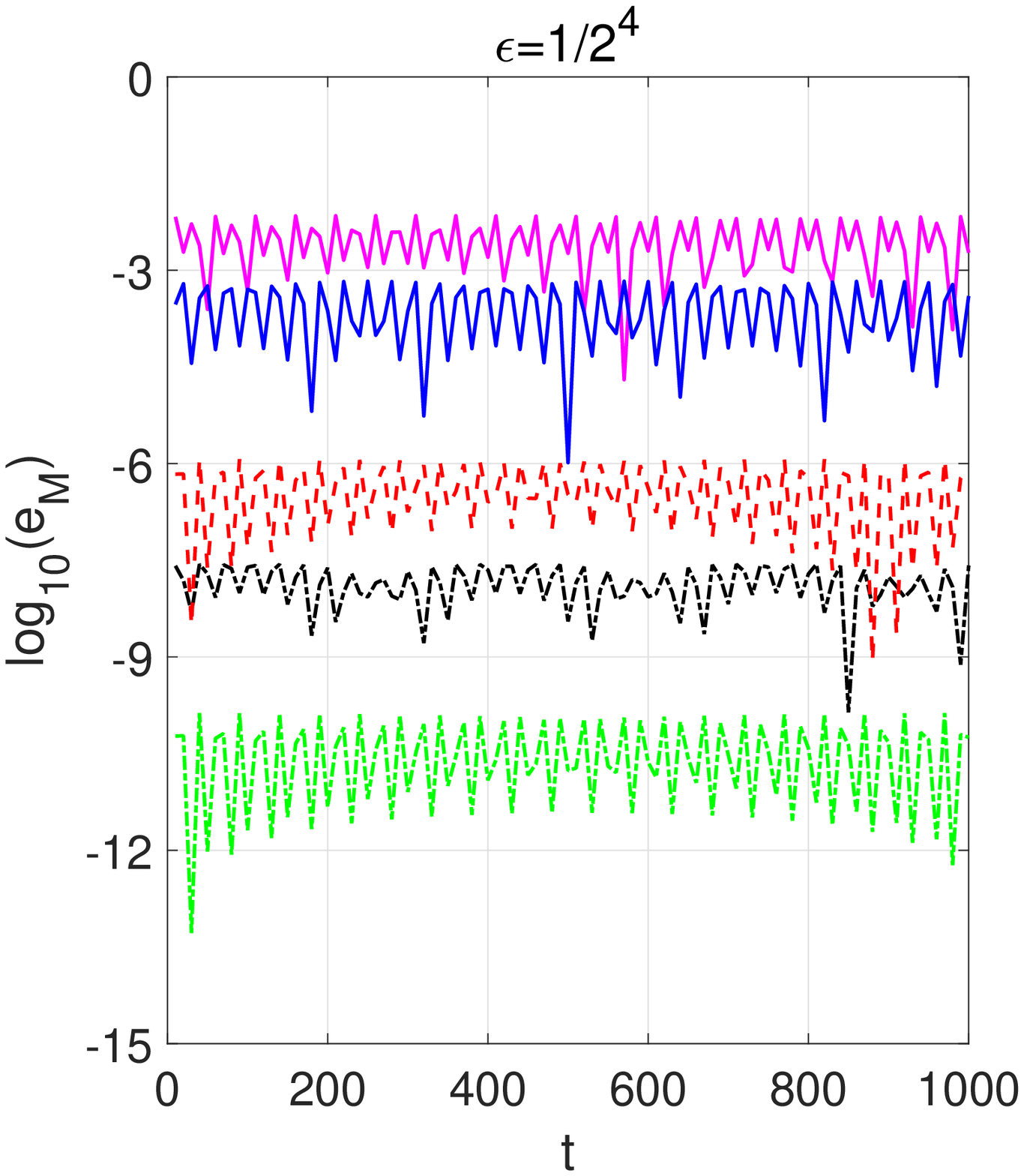}\\
{\small (i)} & {\small (ii)} & {\small (iii)}\\
\end{tabular}
\caption{Problem 1. Evolution of the momentum error $e_{M}:=\frac{|M(x_{n},v_n)-M(x_0,v_0)|}{|M(x_0,v_0)|}$ as function of time  $t=nh$. }
\label{fig:problem13}
\end{figure}

\noindent\vskip3mm \noindent\textbf{Problem 2.}
The second problem  considers the charged-particle dynamics \eqref{charged-particle sts-cons} from \cite{Hairer2018} in the constant magnetic field
 with $ B=\frac{1}{2}(0.9,0.1,1)^{\intercal}$ and the  scalar potential
 $U(x)=x_{1}^{3}-x_{2}^{3}+\frac{1}{5}x_{1}^{4}+x_{2}^{4}+x_{3}^{4}$.
 The initial values are chosen as  $x(0)=(0,1.0,0.1)^{\intercal}, v(0)=(0.09,0.55,0.3)^{\intercal}.$

\begin{figure}[t!]
\centering\tabcolsep=0.4mm
\begin{tabular}
[c]{ccc}%
\includegraphics[width=4.6cm,height=4.6cm]{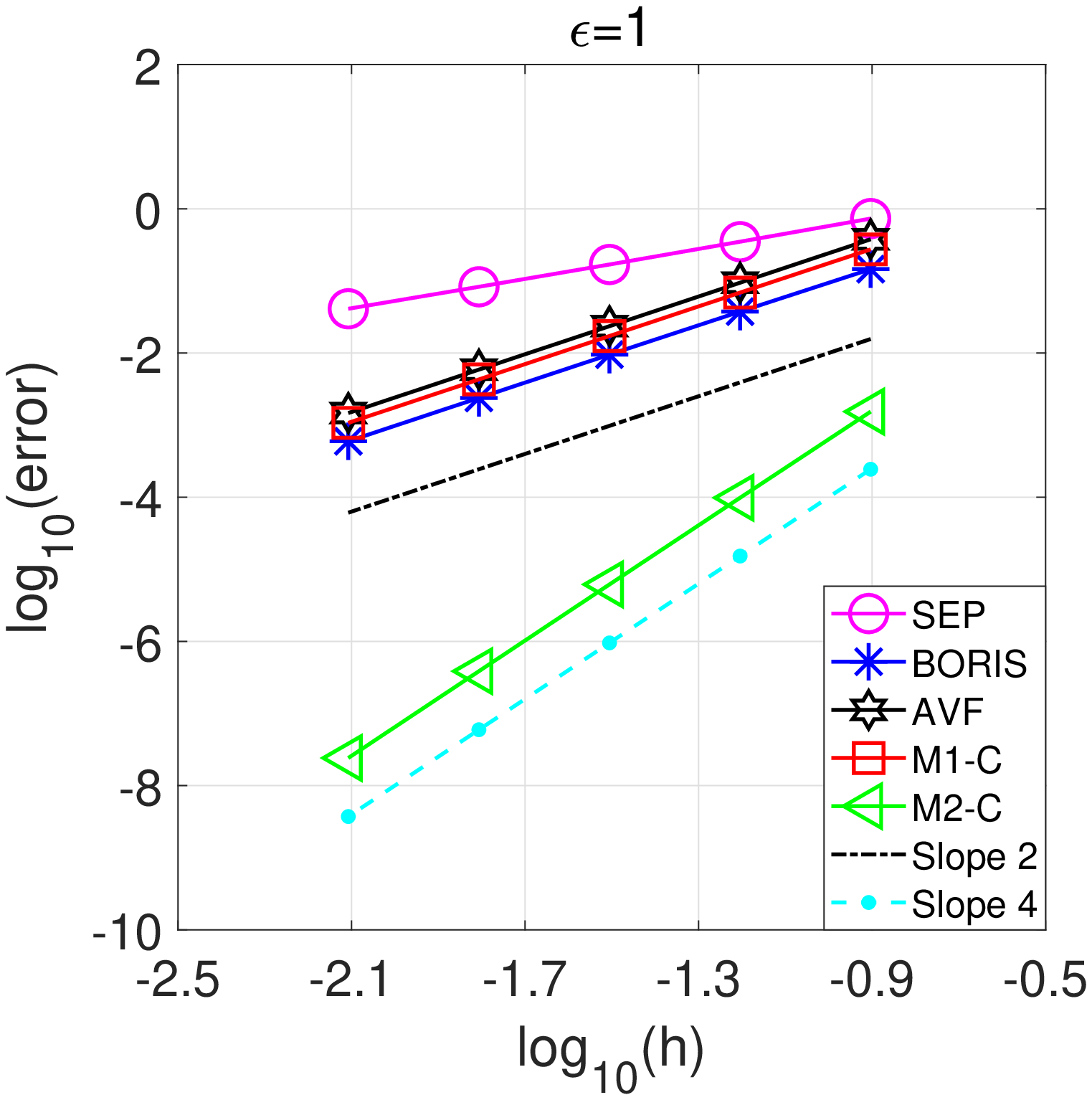} & \includegraphics[width=4.6cm,height=4.6cm]{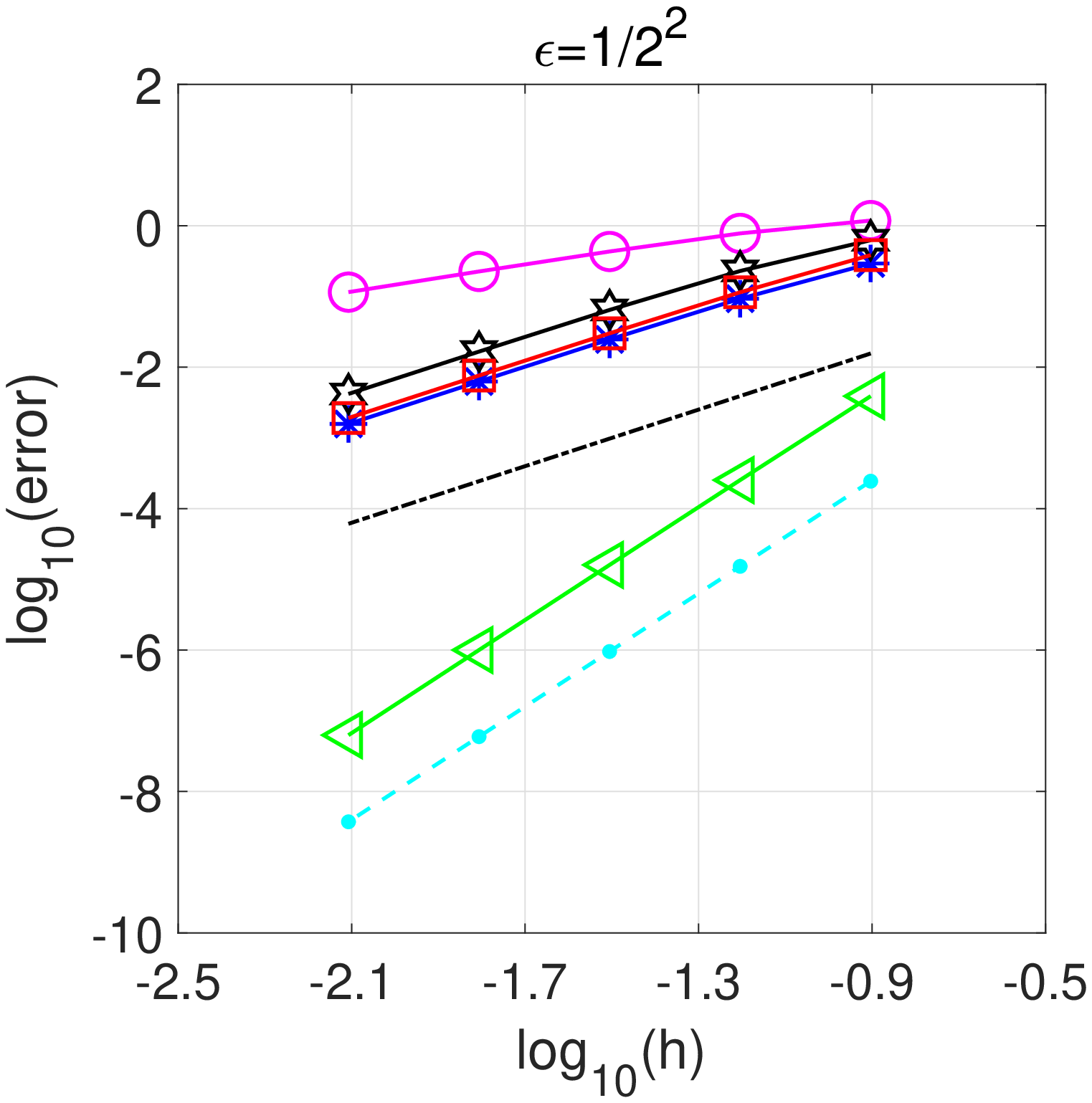} & \includegraphics[width=4.6cm,height=4.6cm]{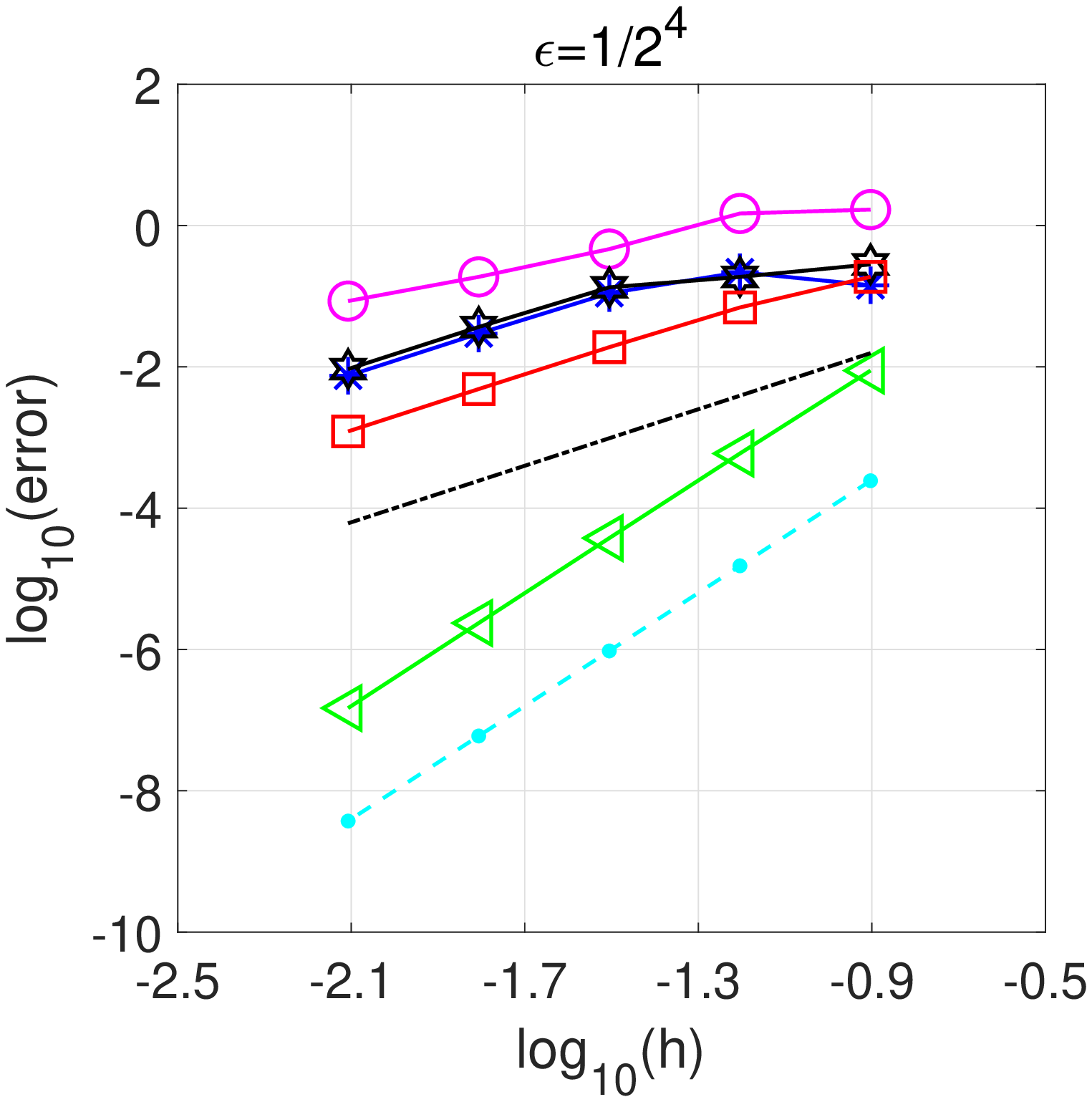}\\
{\small (i)} & {\small (ii)} & {\small (iii)}\\
\end{tabular}
\caption{Problem 2. The global errors $error:=\frac{|x_{n}-x(t_n)|}{|x(t_n)|}+\frac{|v_{n}-v(t_n)|}{|v(t_n)|}$  with $t=10$ and $h=1/2^{k}$ for $k=3,...,7$ under different $\epsilon$. }
\label{fig:problem21}
\end{figure}

\begin{figure}[t!]
\centering\tabcolsep=0.4mm
\begin{tabular}
[c]{ccc}%
\includegraphics[width=4.6cm,height=4.6cm]{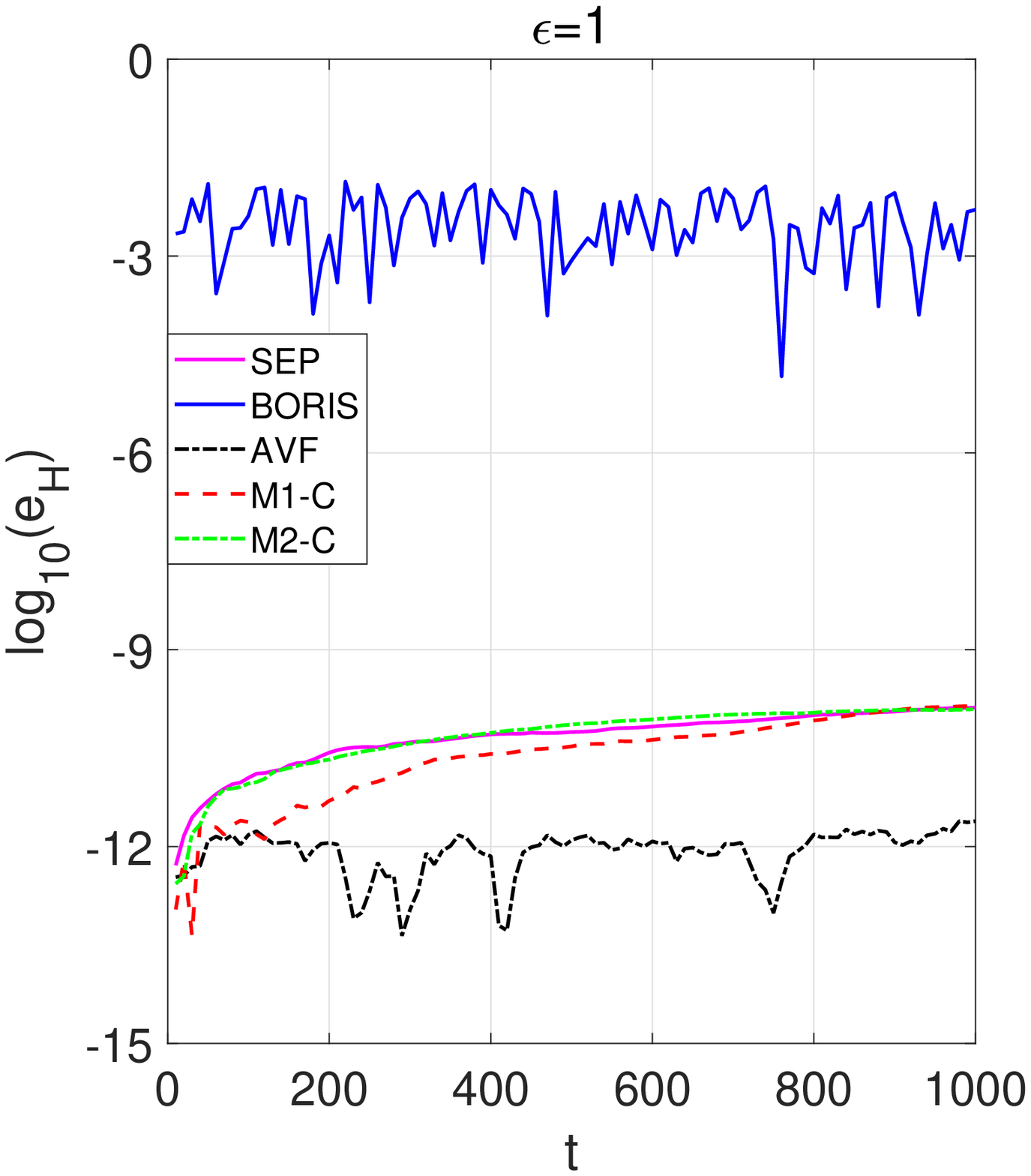} & \includegraphics[width=4.6cm,height=4.6cm]{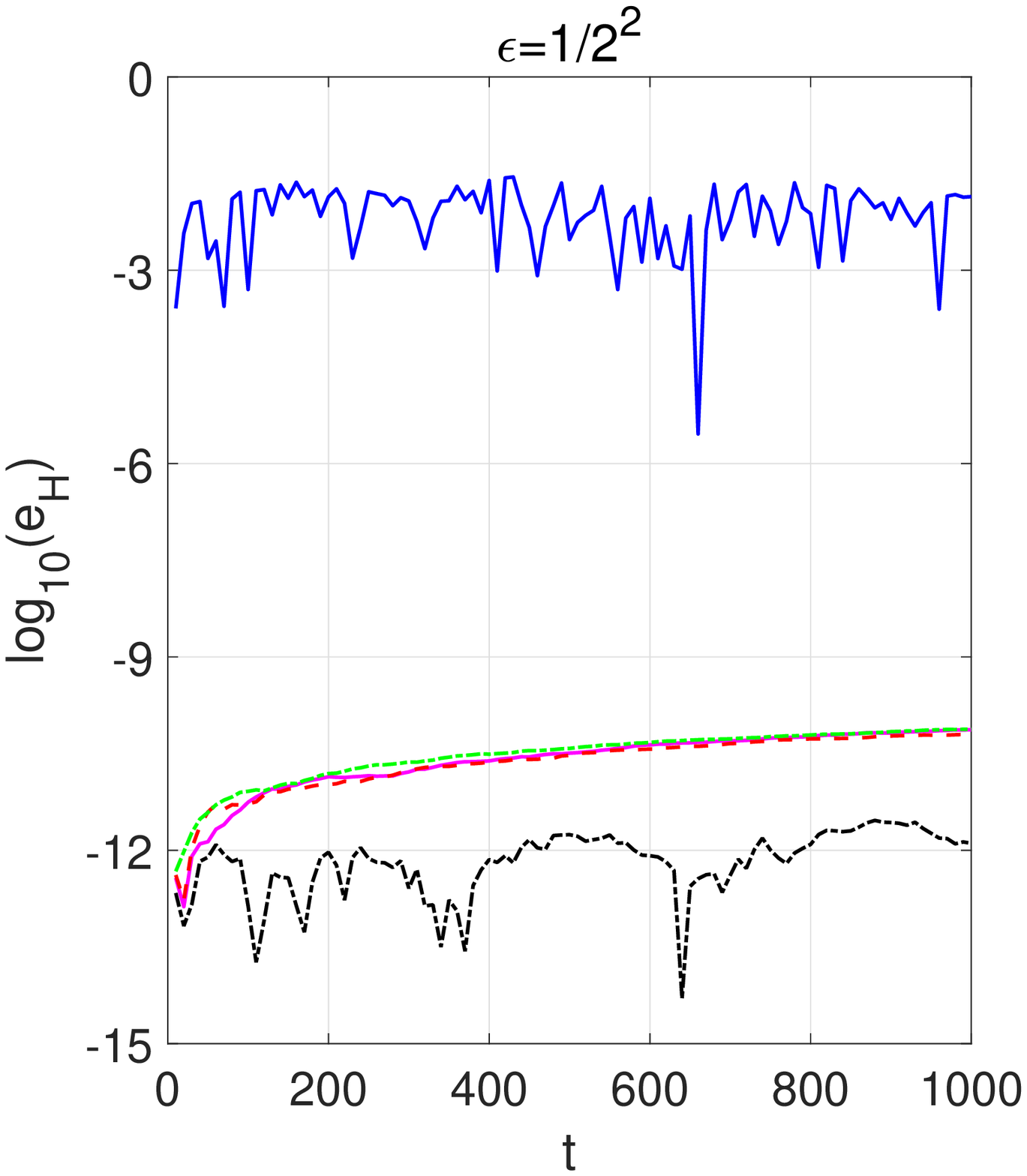} & \includegraphics[width=4.6cm,height=4.6cm]{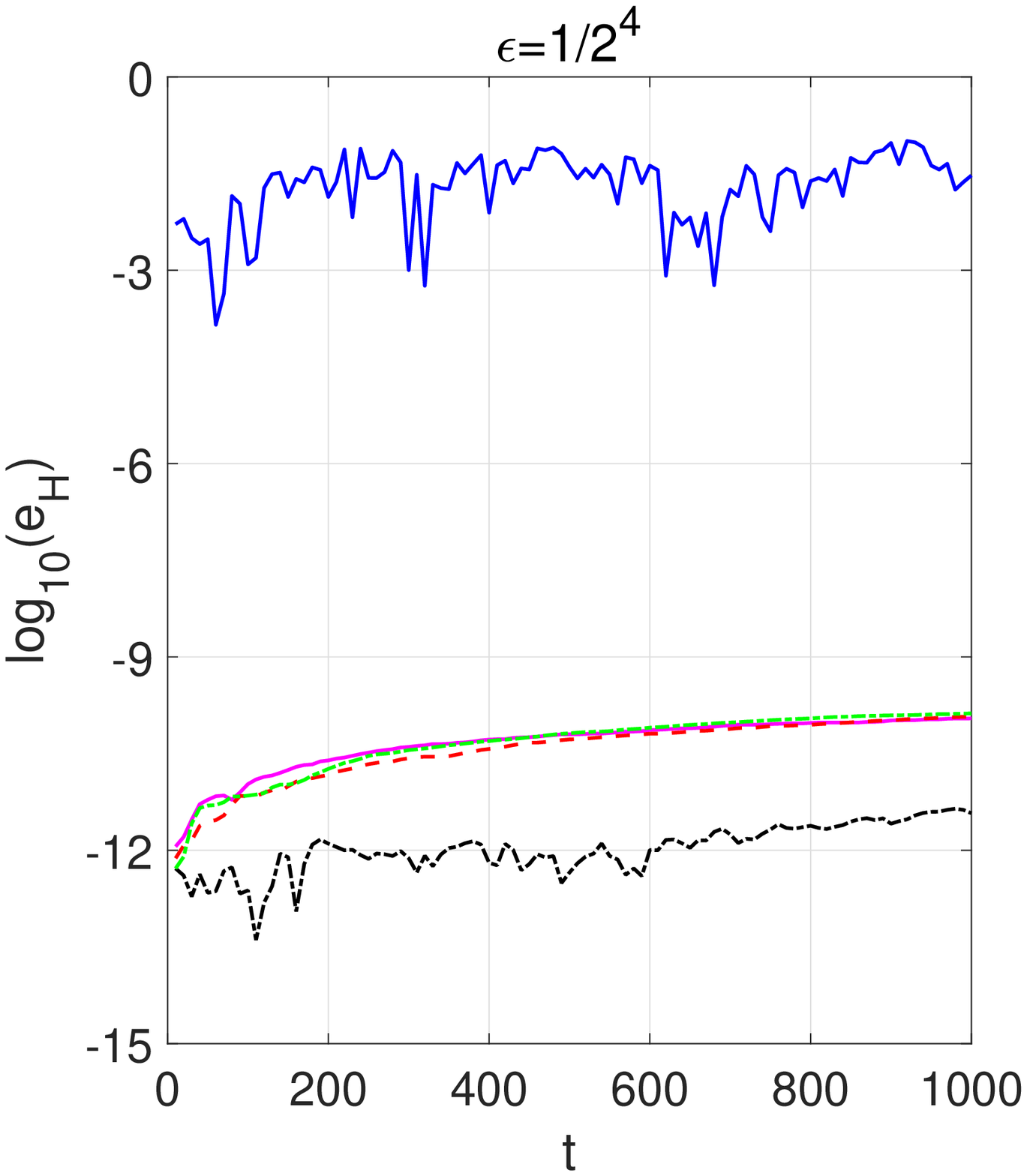}\\
{\small (i)} & {\small (ii)} & {\small (iii)}\\
\end{tabular}
\caption{Problem 2. Evolution of the energy error $e_{H}:=\frac{|H(x_{n},v_n)-H(x_0,v_0)|}{|H(x_0,v_0)|}$ as function of time  $t=nh$.}
\label{fig:problem22}
\end{figure}

\begin{figure}[t!]
\centering\tabcolsep=0.4mm
\begin{tabular}
[c]{ccc}%
\includegraphics[width=4.6cm,height=4.6cm]{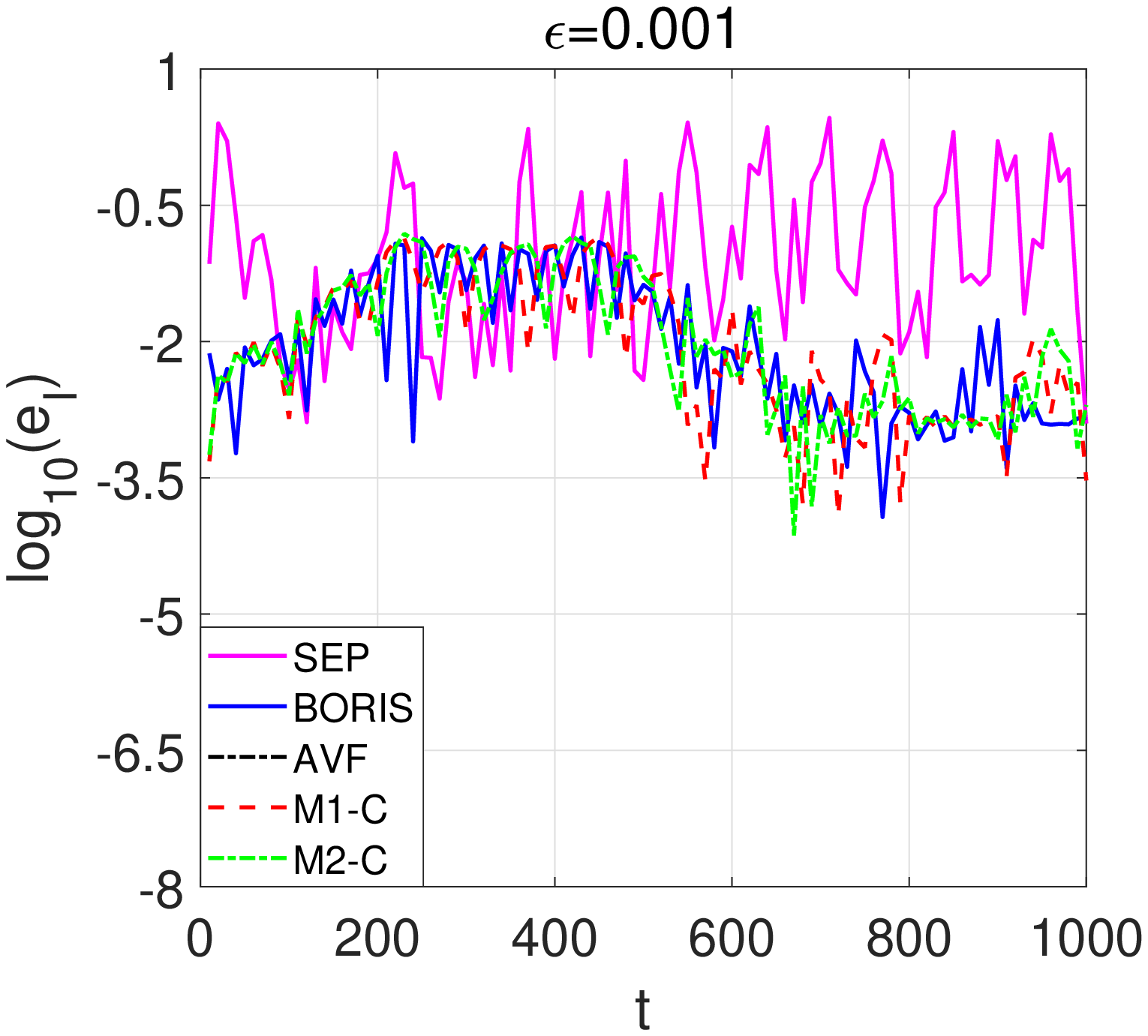} & \includegraphics[width=4.6cm,height=4.6cm]{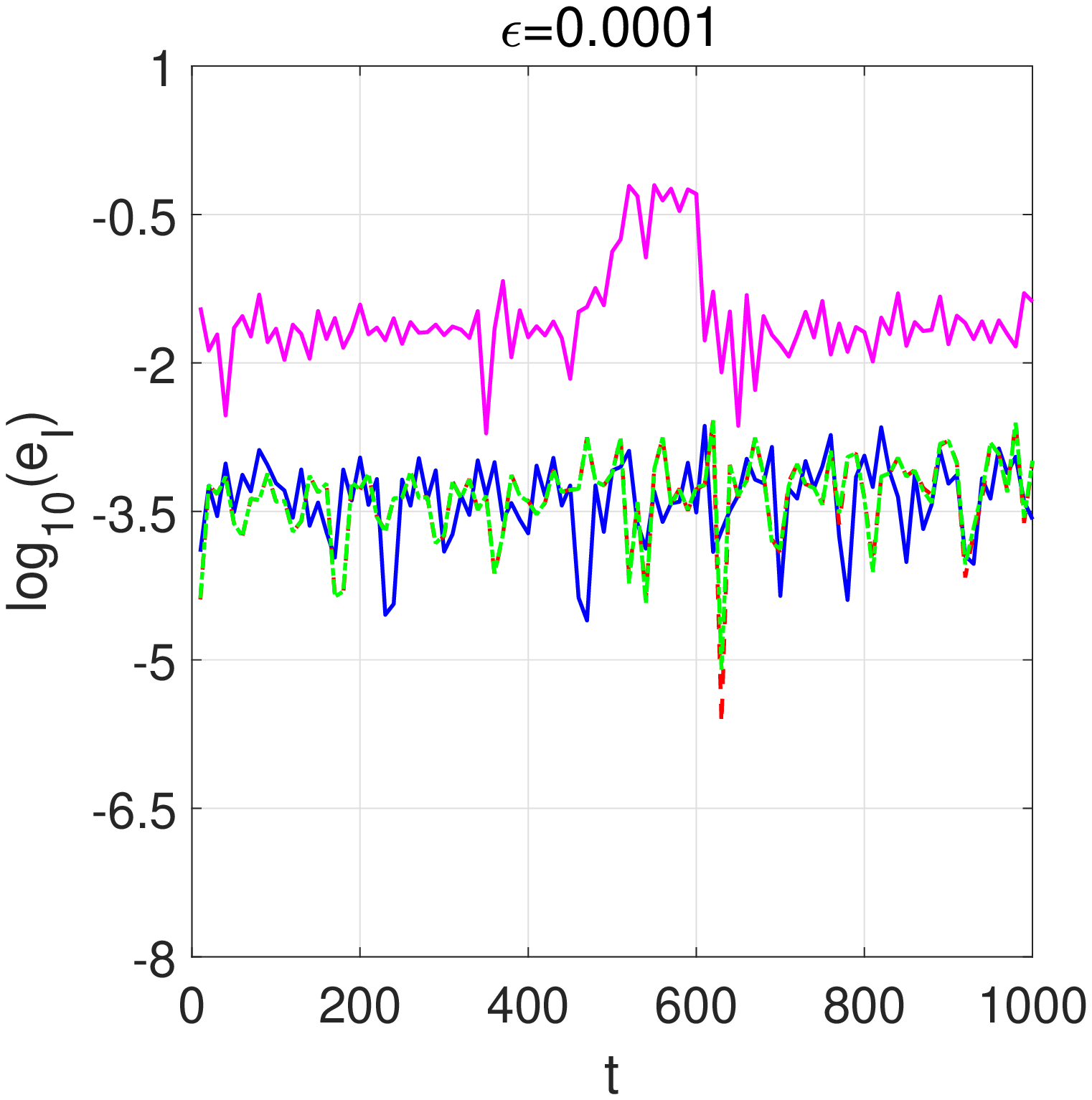} & \includegraphics[width=4.6cm,height=4.6cm]{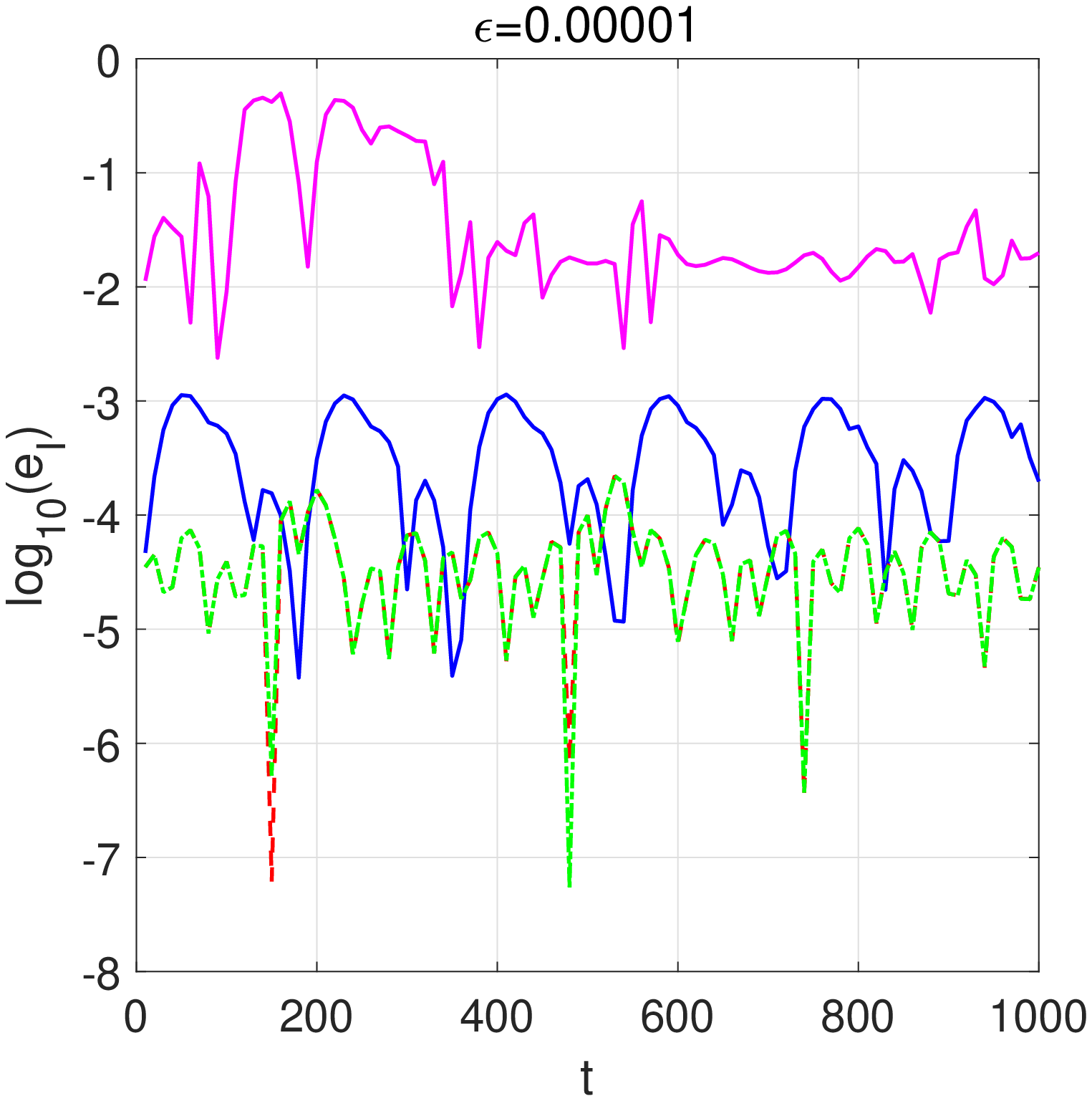}\\
{\small (i)} & {\small (ii)} & {\small (iii)}\\
\end{tabular}
\caption{Problem 2. Evolution of the magnetic moment error $e_{I}:=\frac{|I(x_{n},v_n)-I(x_0,v_0)|}{|I(x_0,v_0)|}$ as function of time $t=nh$.}
\label{fig:problem23}
\end{figure}

The problem  is solved on $[0,10]$ with different   $\epsilon$  and $h=1/2^{k}$, where $k=3,...,7$. See Figure \ref{fig:problem21} for the global errors.  Then we integrate the system with $h=\frac{1}{100}$ on $[0,1000]$. Figure
\ref{fig:problem22} presents the energy conservation. Besides, we consider the magnetic moment $$
I(x,v)=\frac{\frac{1}{2}|\dot{x}\times B(x)|^2}{|B(x)|^3},$$ which is an adiabatic invariant of the system \cite{Arnold97,Benettin94,Hairer2018}.
Its errors $e_{I}:=\frac{|I(x_{n},v_n)-I(x_0,v_0)|}{|I(x_0,v_0)|}$ with $h=\frac{1}{100}$ on $[0,1000]$ are
shown in Figure \ref{fig:problem23}.

From the above two numerical results in Figures \ref{fig:problem11}-\ref{fig:problem23}, we can draw the following conclusions.

1) From Figure \ref{fig:problem11} and Figure \ref{fig:problem21}, we can see the global error lines of our methods M1-C and M2-C are respectively nearly parallel to the lines with slope 2 and slope 4, which shows that our methods M1-C and M2-C are  second order and fourth order, respectively. Moreover, it can be seen that our methods M1-C and M2-C  have better accuracy than the methods Boris, AVF and SEP in the literature.

2) The results in  Figure \ref{fig:problem12} and Figure \ref{fig:problem22} are shown  that the energy-preserving methods AVF, M1-C, M2-C  and SEP have an excellent energy conservation. The boris method does not have such conservation.

3) For the magnetic moment conservation, it can be observed  from  Figure \ref{fig:problem13} that all the methods have a long-term conservation behaviour, and M2-C behaves much better than the others.

4) It can be observed from the  results in    Figure \ref{fig:problem23}  that M1-C and M2-C have a long-term  magnetic moment conservation when $\epsilon$ is small. Other methods do not show this near conservation and since the error of AVF is too large, we do not plot the corresponding results in the figure.

\section{Extension to CPD in a nonuniform magnetic field}\label{sec:new}
In this section, we shall extend the obtained integrators to solve the CPD  \eqref{charged-particle
sts-cons} in a nonuniform magnetic field
 $B(x)=(B_1(x),B_2(x),B_3(x))^{\intercal}$, where $B_i(x): \RR^3\rightarrow \RR$ for $i=1,2,3.$
 \subsection{Integrators and their properties}
\begin{algo} \label{alg:EPS}
		For solving the CPD
\eqref{charged-particle sts-cons}  in a nonuniform  magnetic field $B (x)$, define the following continuous-stage   exponential
  integrator
\begin{equation}\label{CSEEP-B}
\left\{\begin{aligned} &X_{\tau}=x_{n}+ h\tau \varphi_{1}( h\widetilde{K}_n )  v_{n}+h^2 \int_{0}^{1}\tau \varphi_{2}( h\widetilde{K}_n ) F (X_\sigma)d\sigma,\ \ \ 0\leq\tau\leq1,\\
&x_{n+1}=x_{n}+ h\varphi_1( h\widetilde{K}_n ) v_{n}+h^2
\int_{0}^{1}\varphi_{2}( h\widetilde{K}_n )F (X_\tau)d\tau,\\
&v_{n+1}=\varphi_0( h\widetilde{K}_n )v_{n}+h
\int_{0}^{1}\varphi_{1}( h\widetilde{K}_n )F
(X_\tau)d\tau,
\end{aligned}\right.
\end{equation}
where $h$ is the stepsize and $\widetilde{K}_n=\frac{1}{\epsilon}\widetilde{B}(\frac{x_{n}+x_{n+1}}{2})$. We shall refer to this integrator by M1-B.
		
		Based on M1-B denoted by $\Phi_{h}$, we can obtain a Triple Jump splitting scheme \cite{2006Geometric}: $	\Psi_h=\Phi_{a_1h}\circ \Phi_{a_2h}\circ\Phi_{a_3h}$,
		where  $a_1=a_3=\frac{1}{2-\sqrt[3]{2}}$ and $a_2=-\frac{\sqrt[3]{2}}{2-\sqrt[3]{2}}$. We shall call this integrator M2-B.
	\end{algo}

For these two new integrators, they have the following propositions.
\begin{prop}$\bullet$  Both M1-B and M2-B are symmetric and exactly preserve the energy \eqref{energy of cha} of CPD, i.e., $$H(x_n,v_n)=H(x_0,v_0)\ \ \ \textmd{for}\ \ n=1,2,\ldots,T/h.$$
$\bullet$  The integrator M1-B is of order two and M2-B is of order four.
	\end{prop}

\begin{proof}
$\bullet$ By replacing $K$ in the previous proof with $\widetilde{K}_n$ and with the same arguments as before, the symmetry and energy conservation of M1-B can be proved. Then using the symmetric splitting gives the structure conservations of M2-B immediately.

$\bullet$ According to the analysis of \cite{WZ}, it is known that  the convergence of second-order integrator can be obtained by showing the errors when the considered integrator is applied to the  linearized   problem
 \begin{equation*}
\begin{array}[c]{ll}
\ddot{\tilde{x}}(s)=\dot{\tilde{x}}(s) \times   \frac{1}{\epsilon}\widetilde{B}(\frac{x_{n}+x_{n+1}}{2}) +F(\tilde{x}(s)), \quad
\tilde{x}(0)=x(t_n),\quad \dot{\tilde{x}}(t_0)=\dot{x}(t_n),\ \ s\in[0,h],
\end{array}
\end{equation*}
for some $t=t_n+s$ with $n\geq0$. Under this case,  by the variation-of-constants formula  and the convergence analysis stated before, the second order convergence of M1-B can be established. Finally,  Triple Jump splitting scheme acting on M1-B implies the fourth order accuracy of M2-B.

\end{proof}
  \subsection{Numerical tests}

 In what follows, we carry out two numerical tests to show the behaviour of the above two methods. The methods chosen for comparison are still BORIS, AVF and SEP, which are given in Section \ref{sec:tests}.
\noindent\vskip3mm \noindent\textbf{Problem 3.}
We consider the charged-particle dynamics \eqref{charged-particle sts-cons} in the   general magnetic field \cite{Lubich2017}, and the  potential $U(x)$ and field $B(x)$ are given by
$$U(x)=\frac{1}{100\sqrt{x_{1}^2+x_{2}^2}},\ \ \ B(x)=(0,0,\sqrt{x_{1}^2+x_{2}^2})^{\intercal},$$
where $B(x)=\nabla \times A(x)$ with
$A(x)=\frac{1}{3}(-x_2\sqrt{x_{1}^2+x_{2}^2},x_1\sqrt{x_{1}^2+x_{2}^2},0)^{T}.$
 We take the initial values as
$x(0)=(0,1,0.1)^{\intercal},\ v(0)=(0.09,0.05,0.2)^{\intercal}.$
  The problem is solved on $[0,10]$ with $h=1/2^k,$ where $k=3,\ldots,7$.
The global errors  are presented  in Figure
 \ref{fig:problem31}. Then we solve the system
with  a step size $h=\frac{1}{100}$ on the interval $ [0,1000]$. The results of energy and momentum  are
displayed in  Figure \ref{fig:problem32} and Figure \ref{fig:problem33}, respectively.

\begin{figure}[t!]
\centering\tabcolsep=0.4mm
\begin{tabular}
[c]{ccc}%
\includegraphics[width=4.6cm,height=4.6cm]{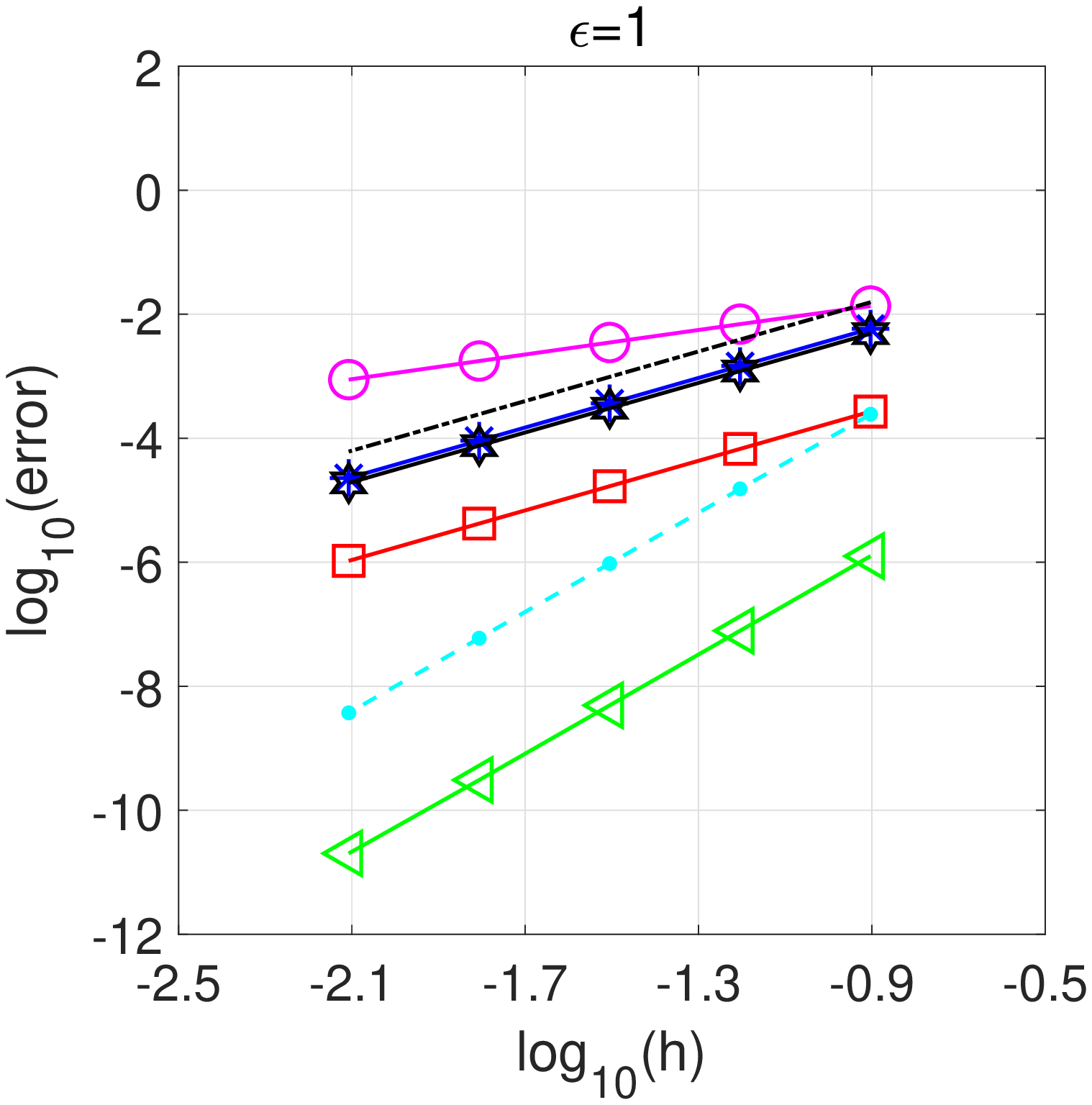} & \includegraphics[width=4.6cm,height=4.6cm]{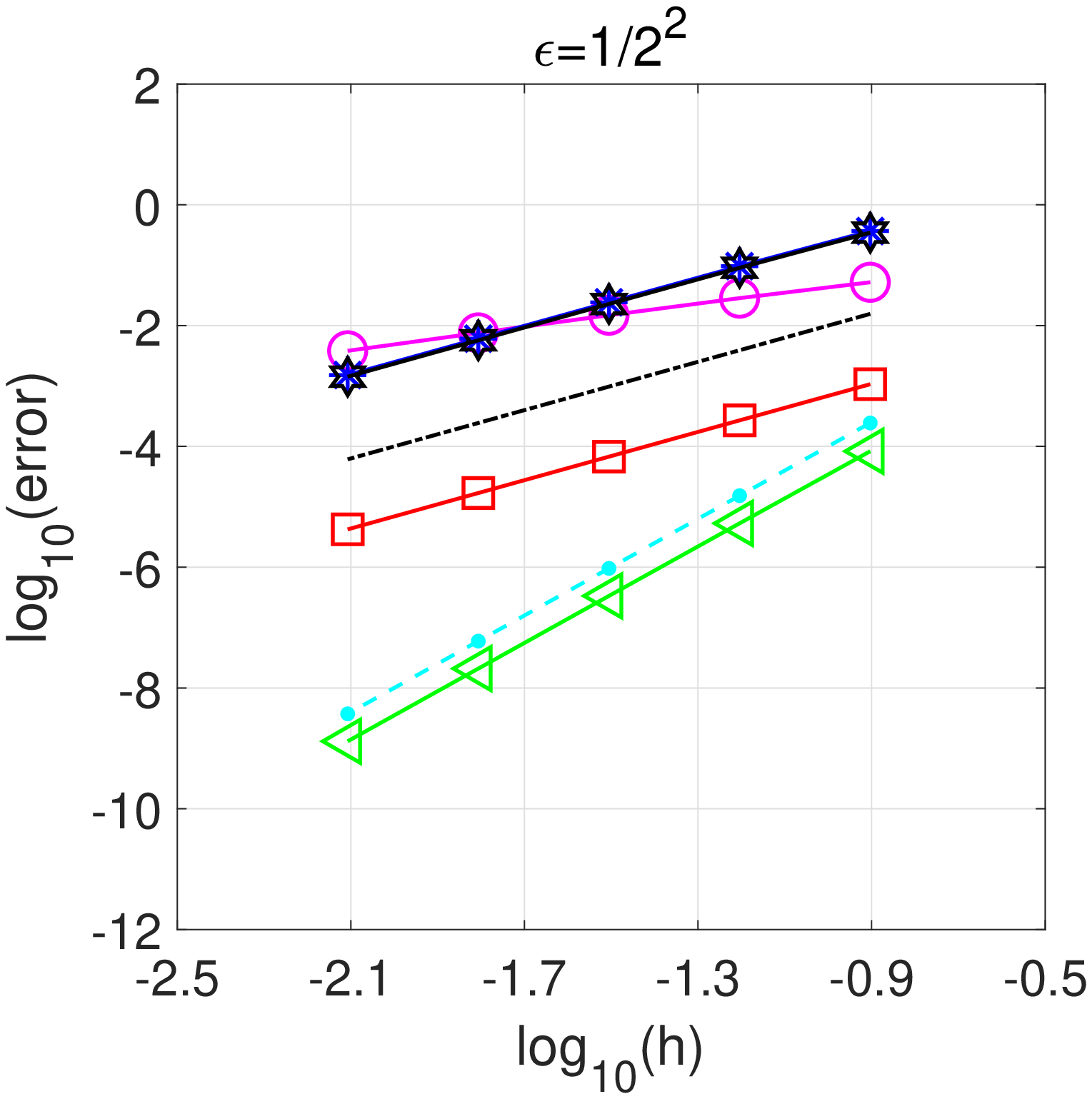}&\includegraphics[width=4.6cm,height=4.6cm]{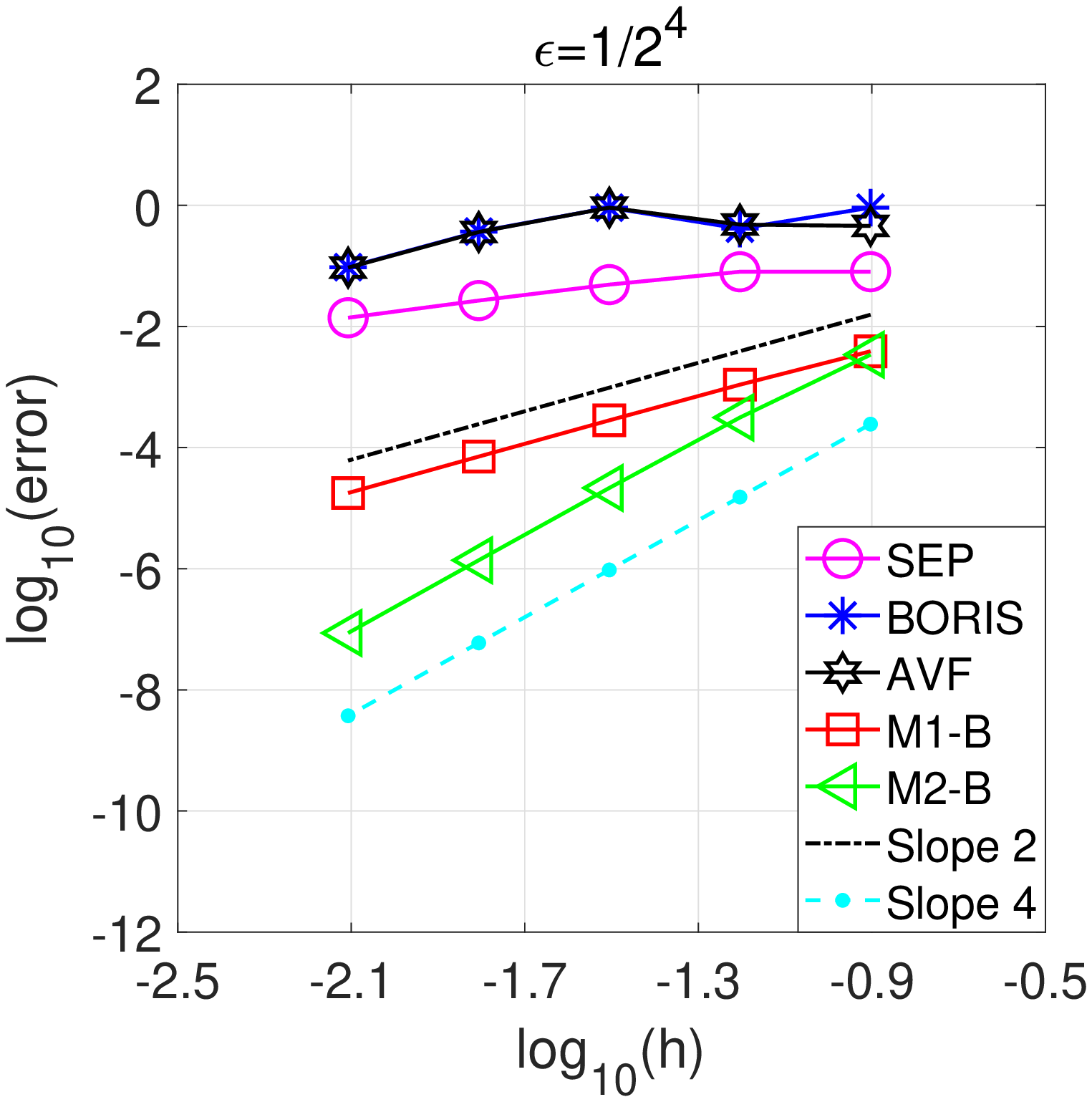}\\
{\small (i)} & {\small (ii)} & {\small (iii)}\\
\end{tabular}
\caption{Problem 3.  The global errors $error:=\frac{|x_{n}-x(t_n)|}{|x(t_n)|}+\frac{|v_{n}-v(t_n)|}{|v(t_n)|}$  with $t=10$ and $h=1/2^{k}$ for $k=3,...,7$ under different $\epsilon$. }
\label{fig:problem31}
\end{figure}

\begin{figure}[t!]
\centering\tabcolsep=0.4mm
\begin{tabular}
[c]{ccc}%
\includegraphics[width=4.6cm,height=4.6cm]{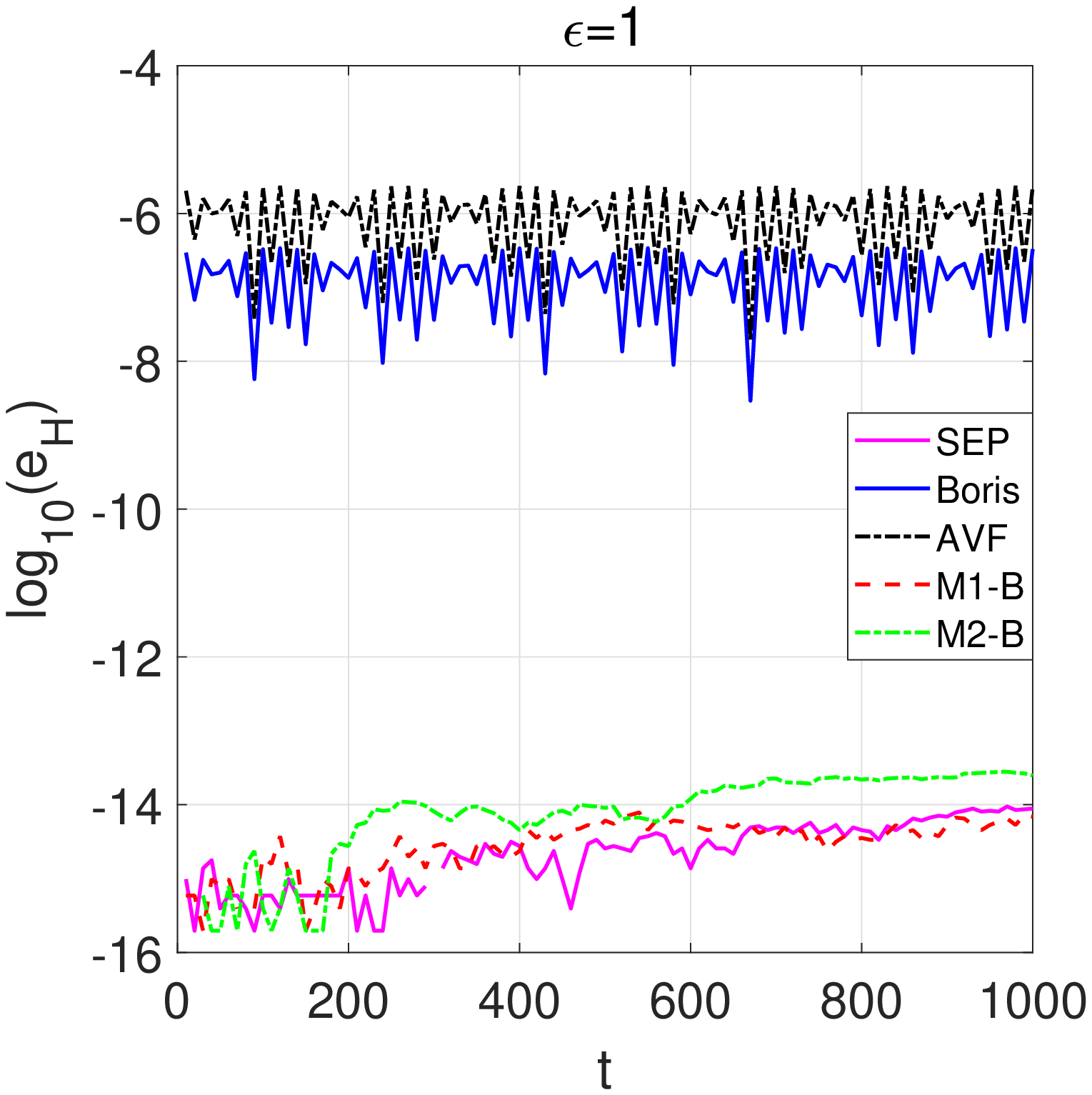} & \includegraphics[width=4.6cm,height=4.6cm]{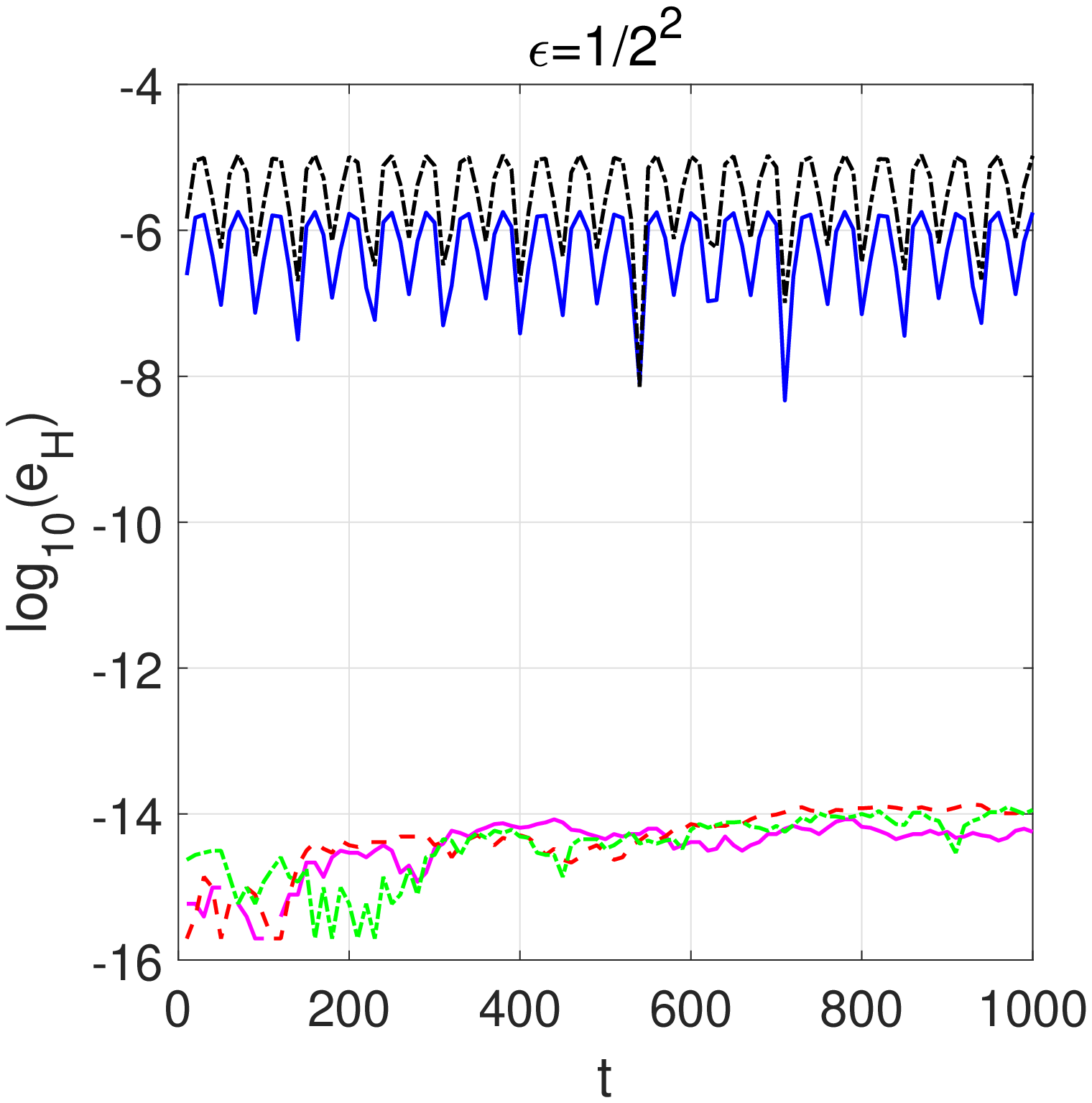} & \includegraphics[width=4.6cm,height=4.6cm]{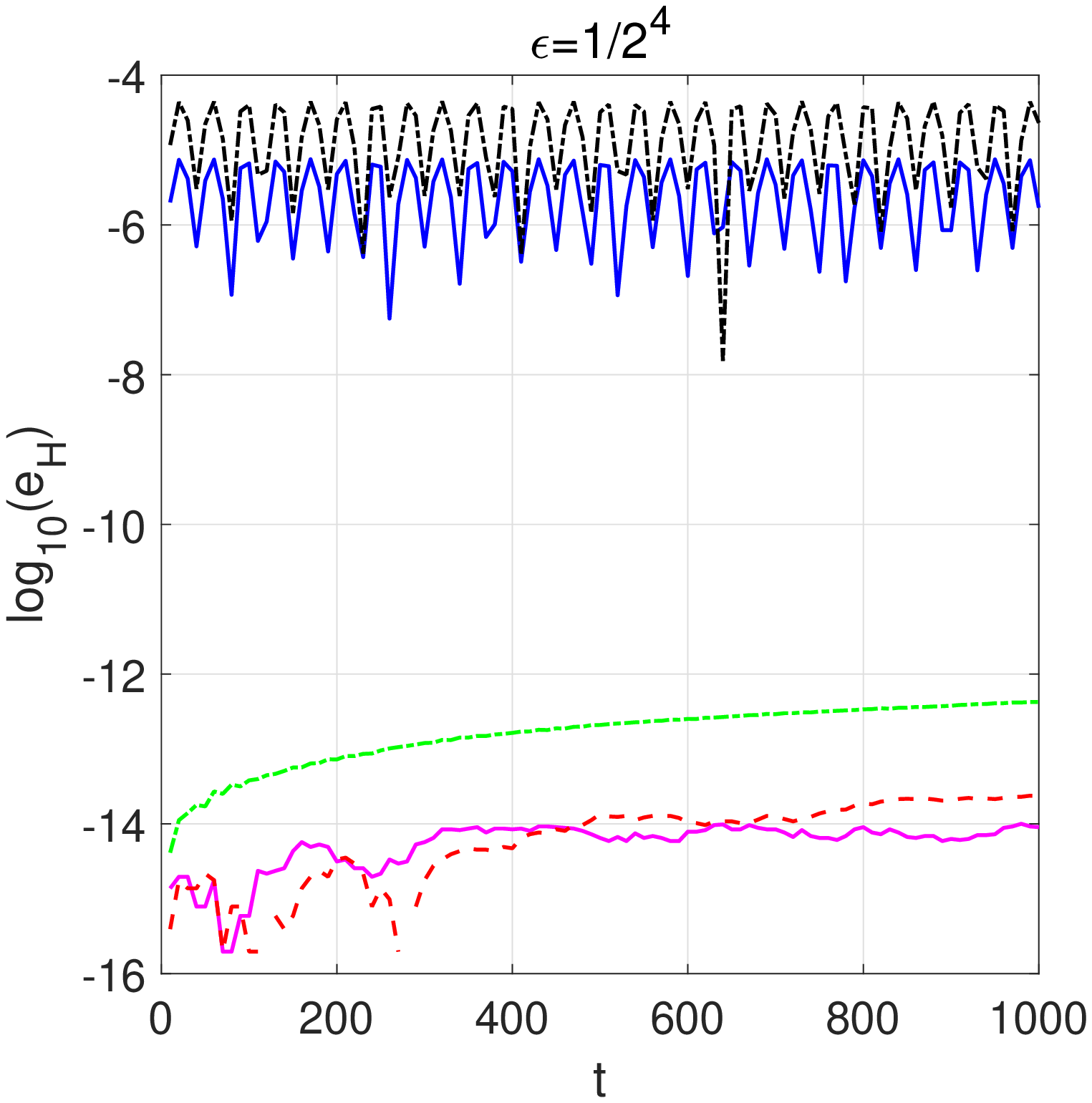}\\
{\small (i)} & {\small (ii)} & {\small (iii)}\\
\end{tabular}
\caption{Problem 3.  Evolution of the energy error $e_{H}:=\frac{|H(x_{n},v_n)-H(x_0,v_0)|}{|H(x_0,v_0)|}$ as function of time  $t=nh$.}
\label{fig:problem32}
\end{figure}

\begin{figure}[t!]
\centering\tabcolsep=0.4mm
\begin{tabular}
[c]{ccc}%
\includegraphics[width=4.6cm,height=4.6cm]{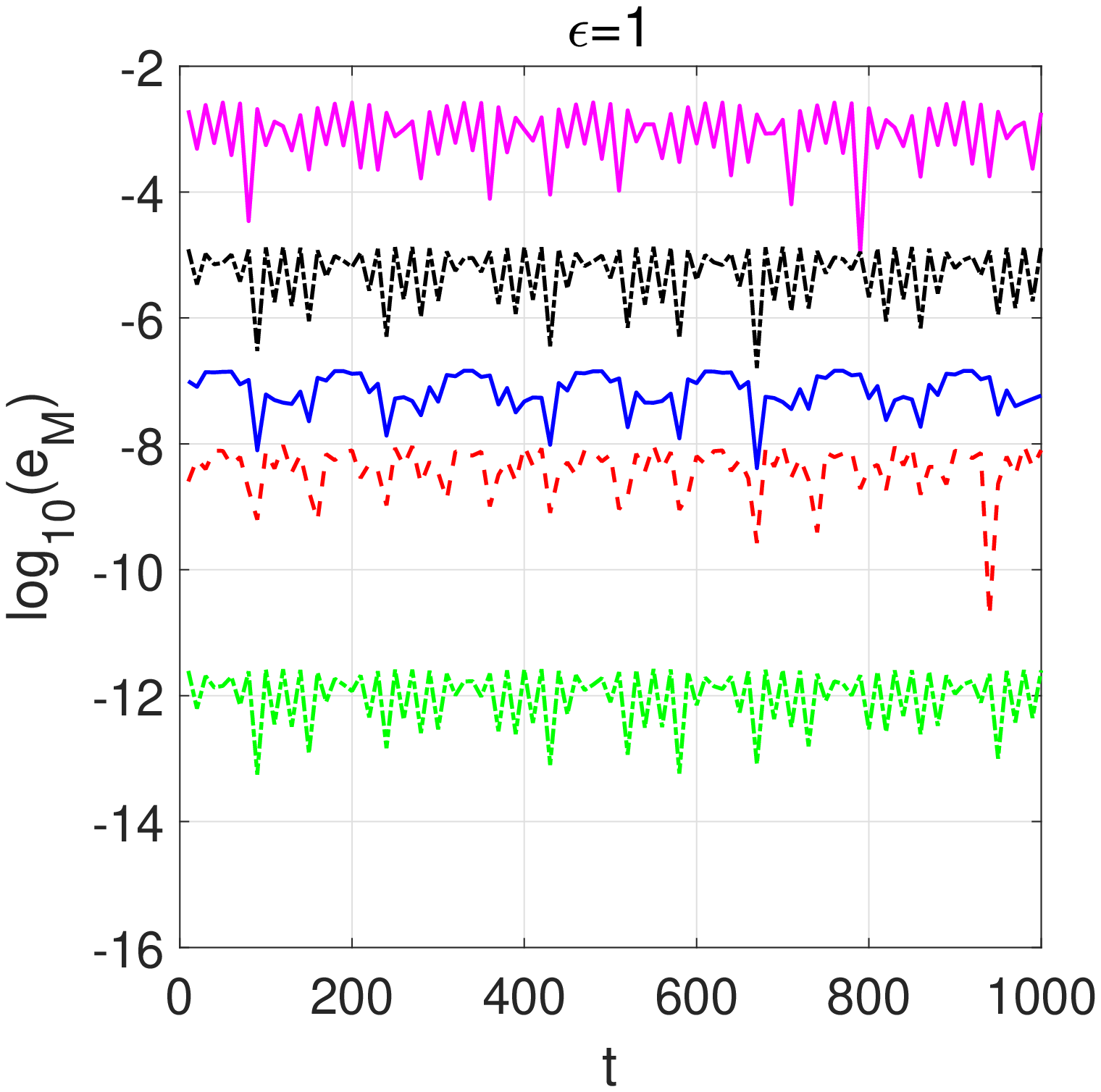} & \includegraphics[width=4.6cm,height=4.6cm]{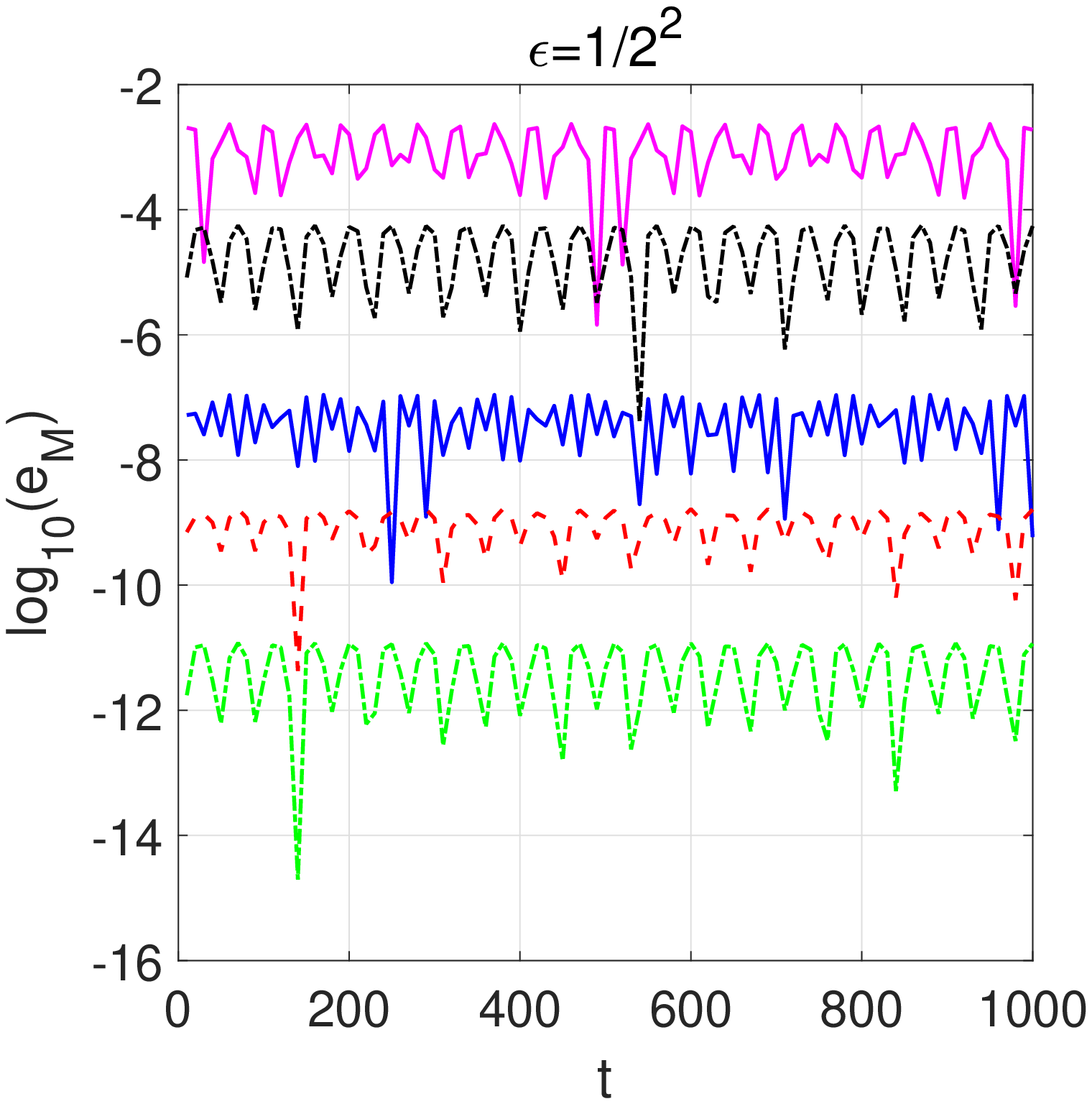}&\includegraphics[width=4.6cm,height=4.6cm]{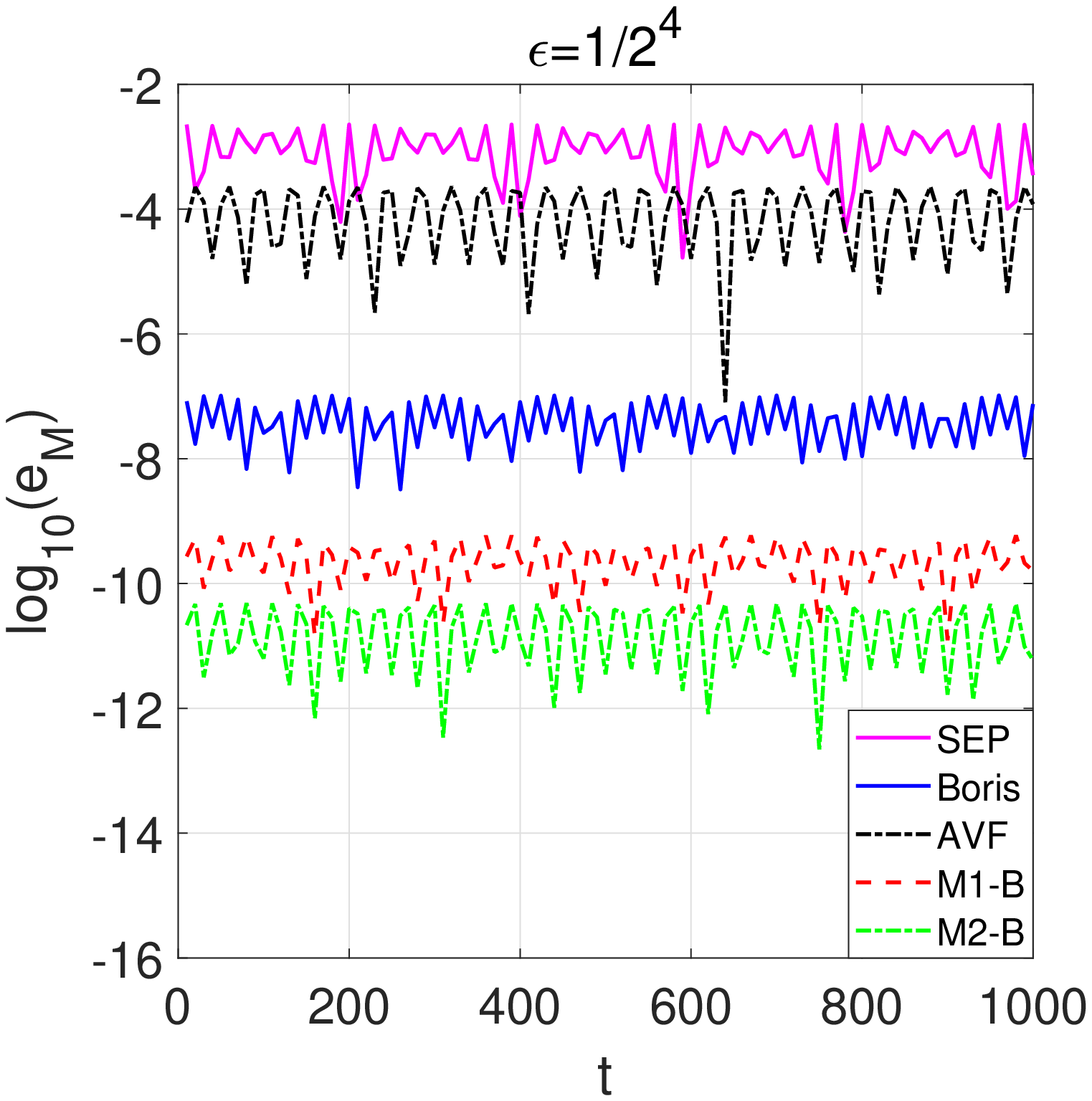}\\
{\small (i)} & {\small (ii)} & {\small (iii)}\\
\end{tabular}
\caption{Problem 3.  Evolution of the momentum error $e_{M}:=\frac{|M(x_{n},v_n)-M(x_0,v_0)|}{|M(x_0,v_0)|}$ as function of time  $t=nh$.}
\label{fig:problem33}
\end{figure}

\noindent\vskip3mm \noindent\textbf{Problem 4.}
The last test is devoted to   charged-particle dynamics \eqref{charged-particle sts-cons} with the general magnetic field \cite{Hairer2018}
 $$B(x)=\nabla \times \frac{1}{4}(x_{3}^2-x_{2}^2,x_{3}^2-x_{1}^2,x_{2}^2-x_{1}^2)^{\intercal}= \frac{1}{2}(x_{2}-x_{3},x_{1}+x_{3},x_{2}-x_{1})^{\intercal},$$
 and the scalar potential
 $U(x)=x_{1}^{3}-x_{2}^{3}+\frac{1}{5}x_{1}^{4}+x_{2}^{4}+x_{3}^{4}.$
The initial values are chosen as $x(0)=(0.0,1.0,0.1)^{\intercal}, \ v(0)=(0.09,0.55,0.30)^{\intercal}.$
This system  is integrated on $[0,10]$ with different   $\epsilon$  and $h=1/2^{k}$, where $k=3,...,7$, and see Figure \ref{fig:problem41} for the errors. Then we solve the problem with $h=\frac{1}{100}$ on  $[0,1000]$. Figure
\ref{fig:problem42} and Figure
\ref{fig:problem43} display the errors $e_{H}$ of the energy and   $e_{I}$ of the magnetic moment, respectively.

\begin{figure}[t!]
\centering\tabcolsep=0.4mm
\begin{tabular}
[c]{ccc}%
\includegraphics[width=4.6cm,height=4.6cm]{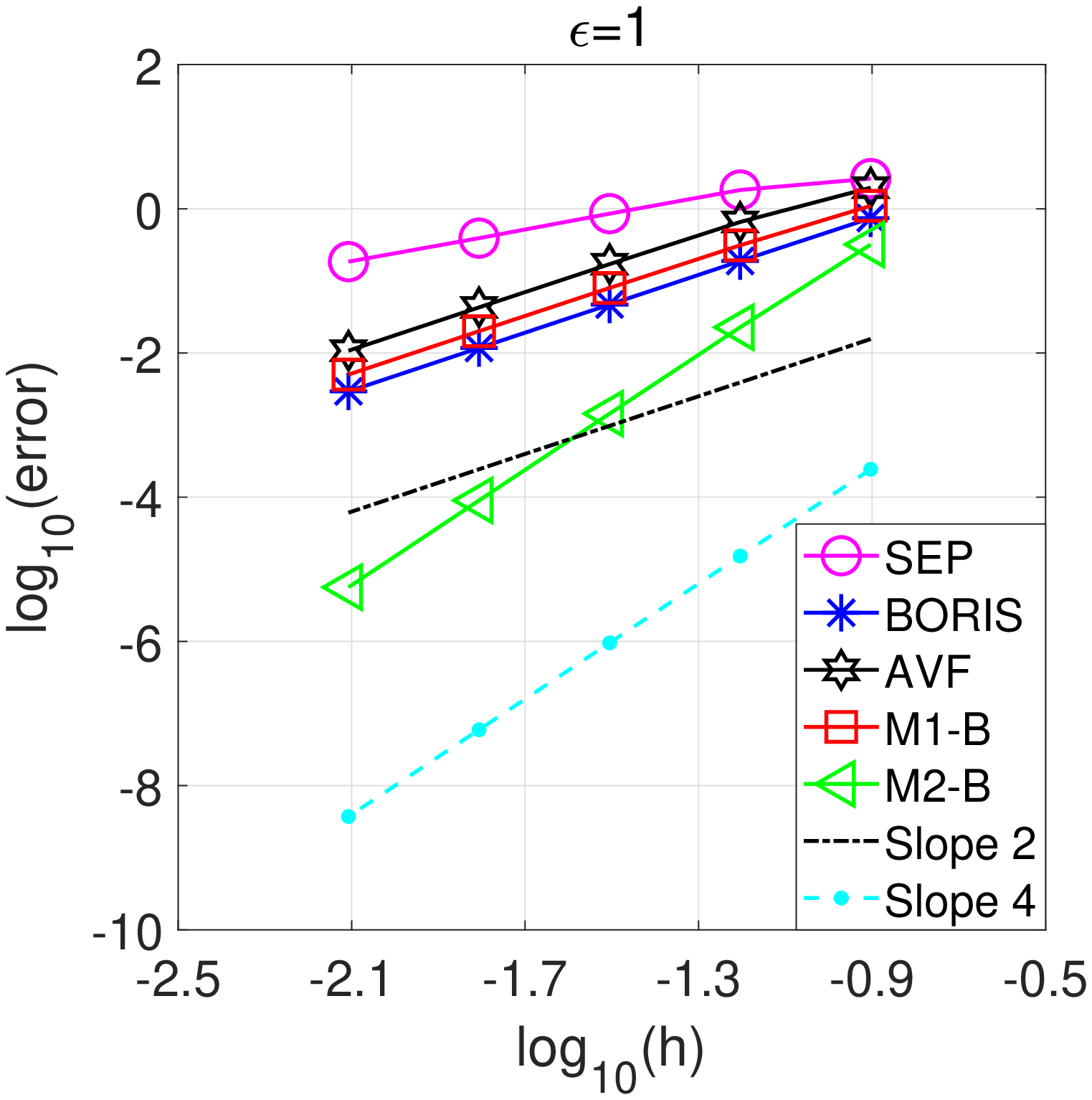} & \includegraphics[width=4.6cm,height=4.6cm]{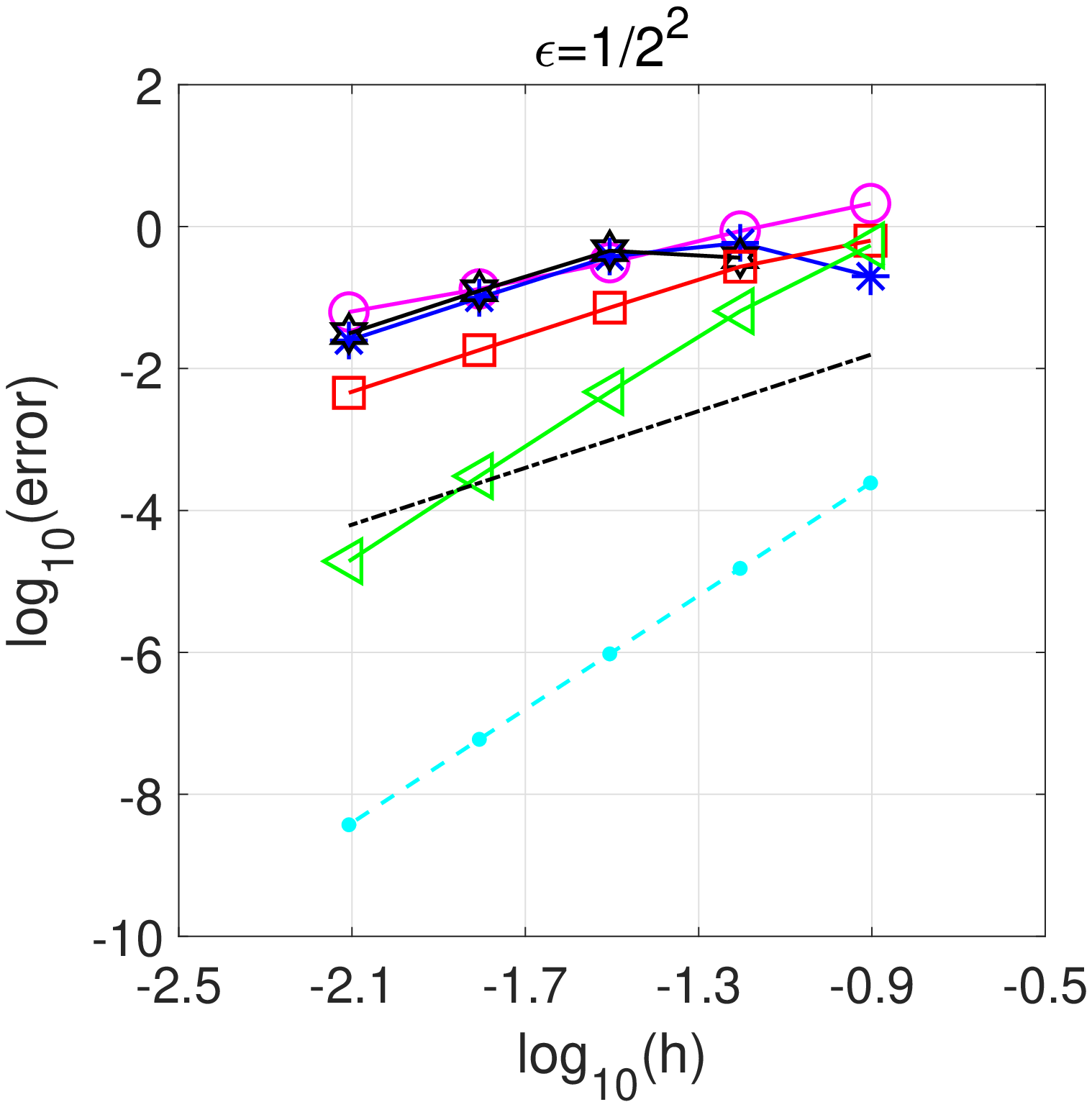} & \includegraphics[width=4.6cm,height=4.6cm]{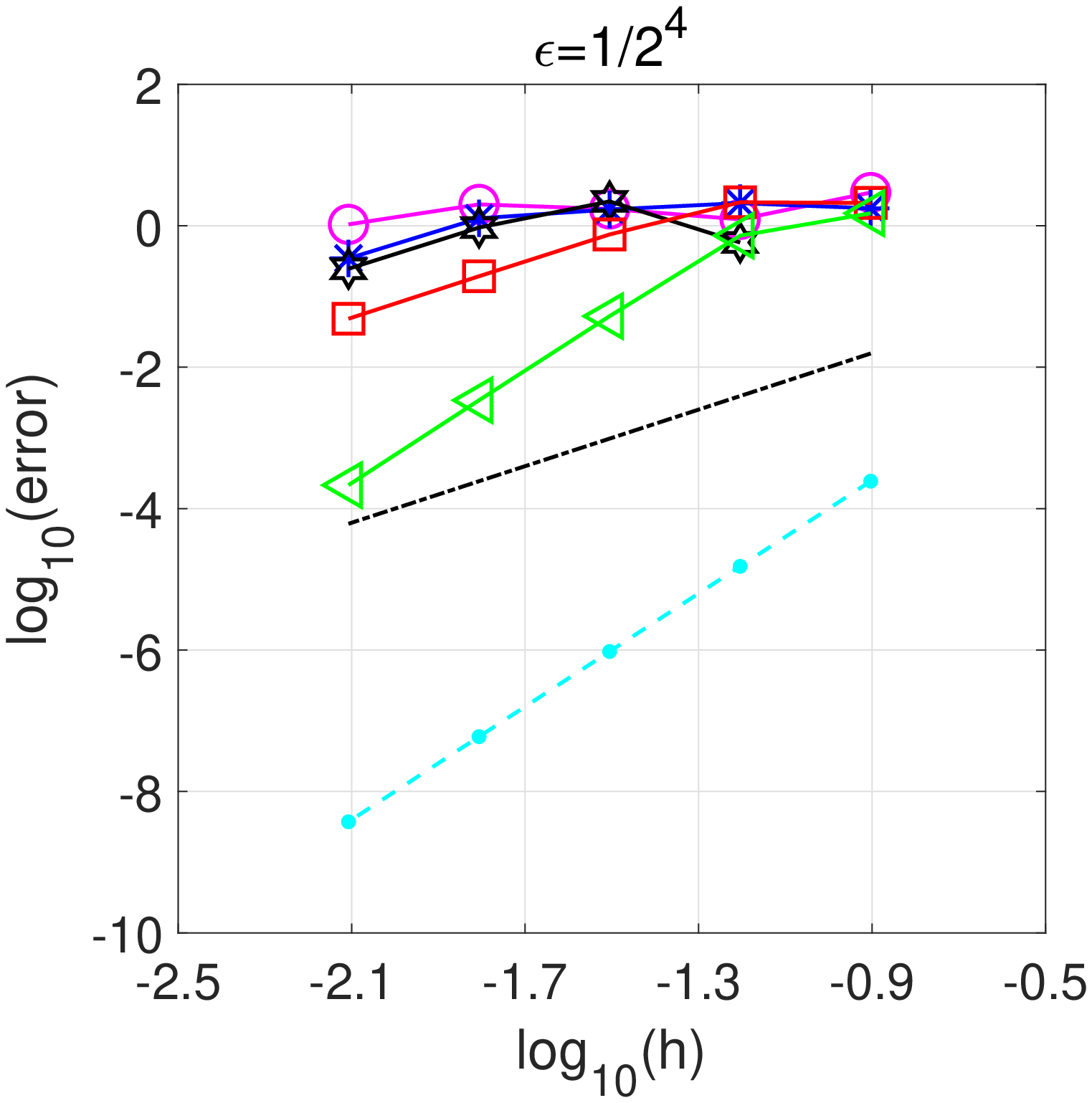}\\
{\small (i)} & {\small (ii)} & {\small (iii)}\\
\end{tabular}
\caption{Problem 4.  The global errors $error:=\frac{|x_{n}-x(t_n)|}{|x(t_n)|}+\frac{|v_{n}-v(t_n)|}{|v(t_n)|}$  with $t=10$ and $h=1/2^{k}$ for $k=3,...,7$ under different $\epsilon$. }
\label{fig:problem41}
\end{figure}

\begin{figure}[t!]
\centering\tabcolsep=0.4mm
\begin{tabular}
[c]{ccc}%
\includegraphics[width=4.6cm,height=4.6cm]{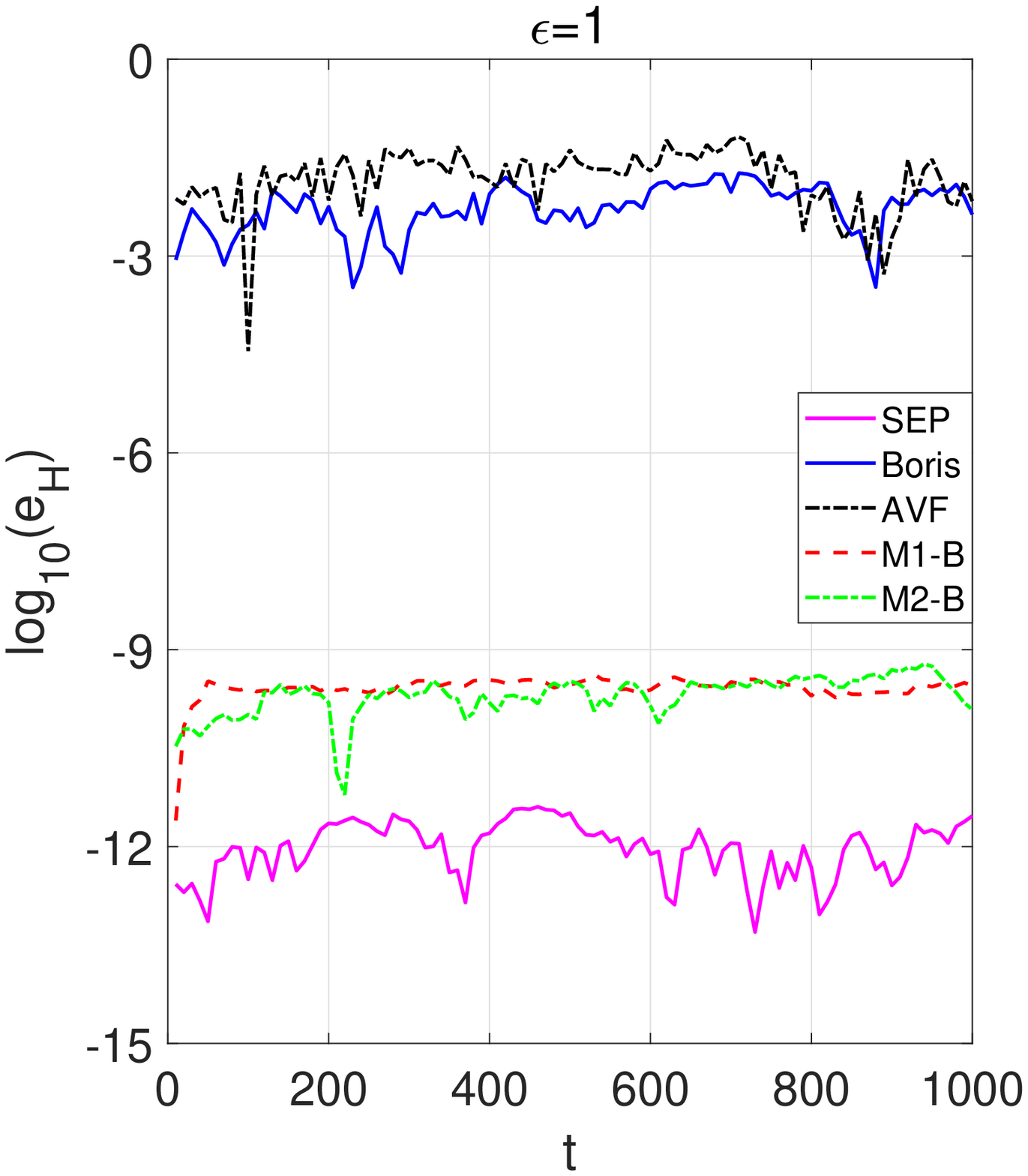} & \includegraphics[width=4.6cm,height=4.6cm]{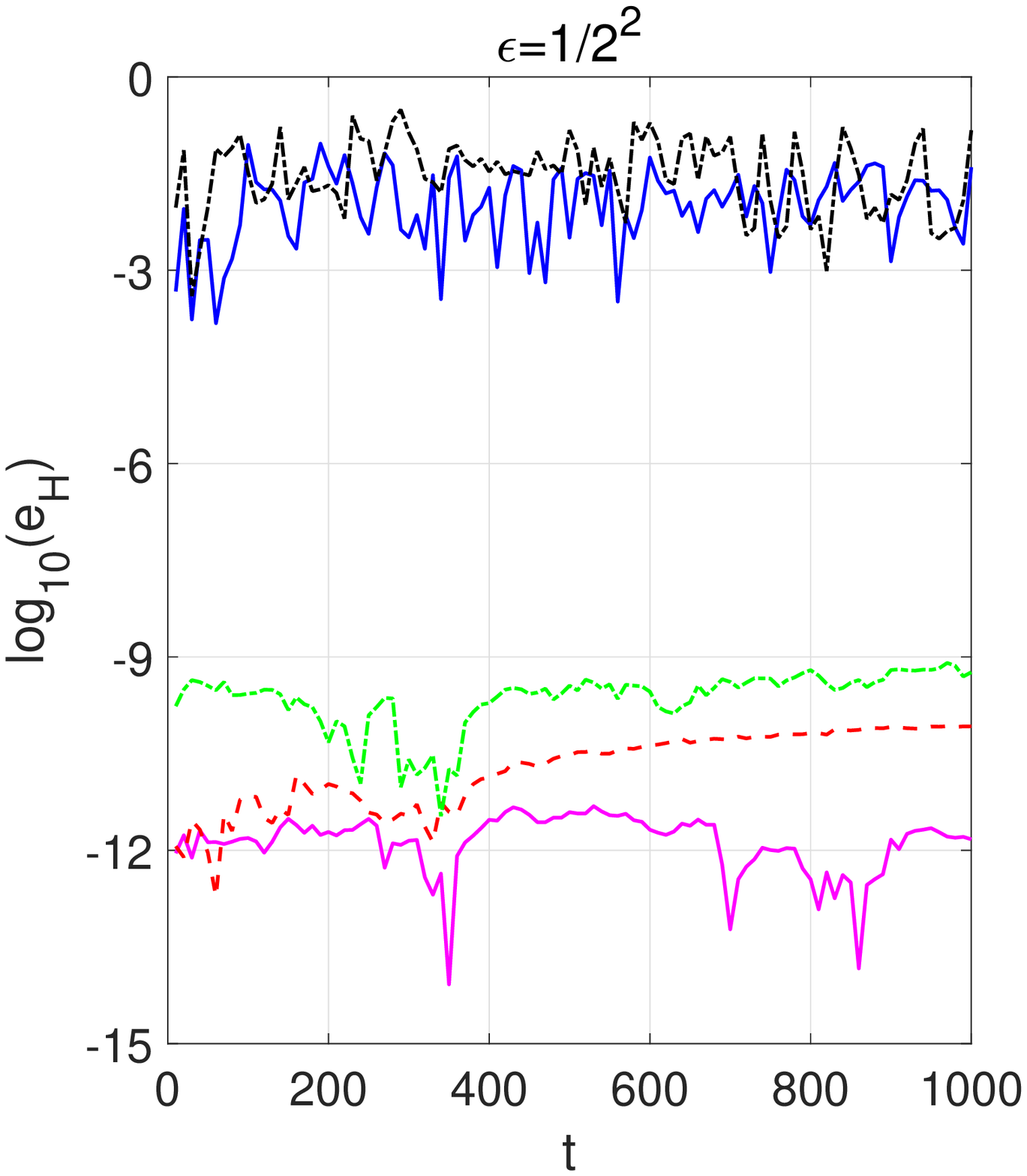} & \includegraphics[width=4.6cm,height=4.6cm]{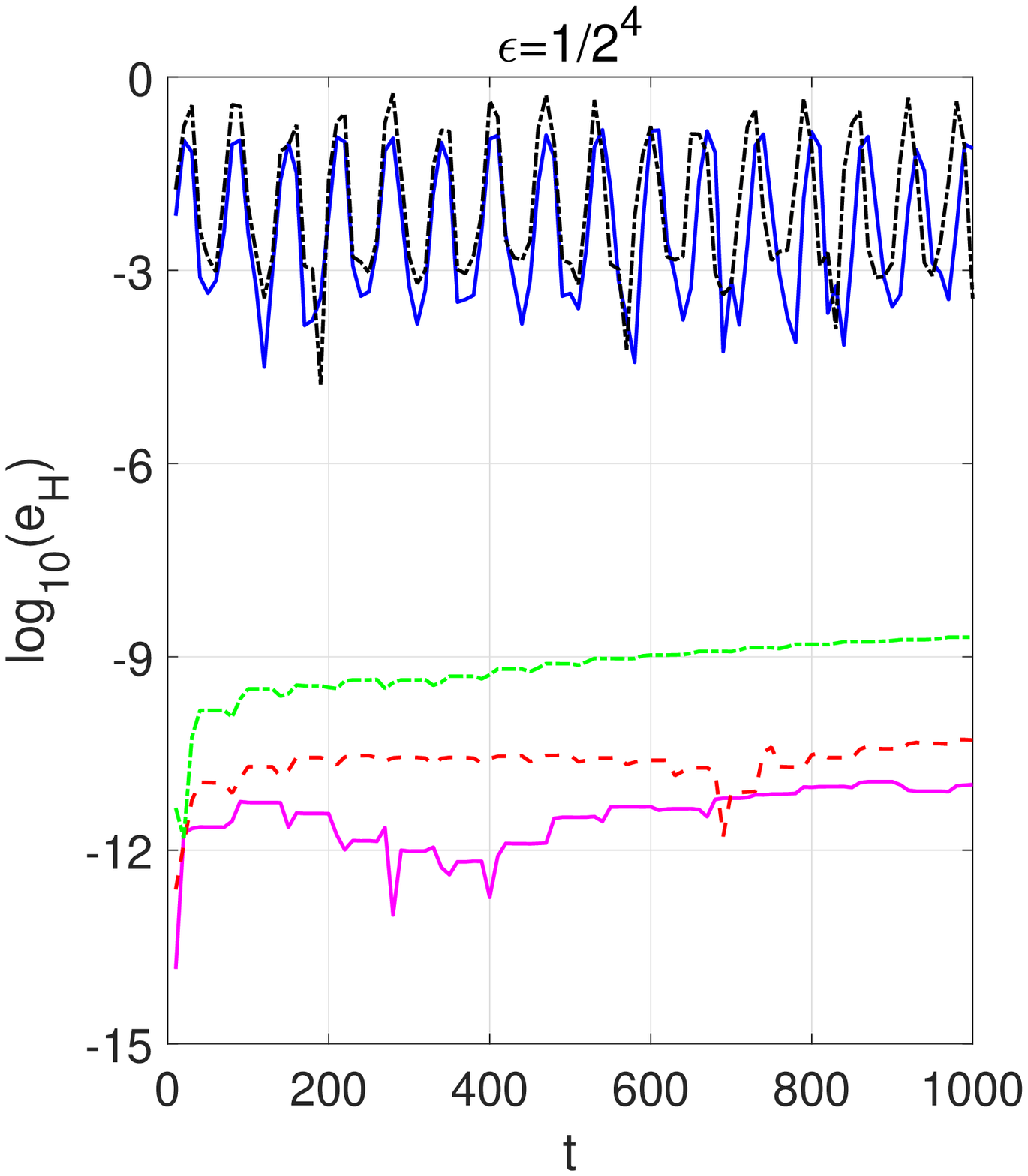}\\
{\small (i)} & {\small (ii)} & {\small (iii)}\\
\end{tabular}
\caption{Problem 4. Evolution of the energy error $e_{H}:=\frac{|H(x_{n},v_n)-H(x_0,v_0)|}{|H(x_0,v_0)|}$ as function of time  $t=nh$.}
\label{fig:problem42}
\end{figure}

\begin{figure}[t!]
\centering\tabcolsep=0.4mm
\begin{tabular}
[c]{ccc}%
\includegraphics[width=4.6cm,height=4.6cm]{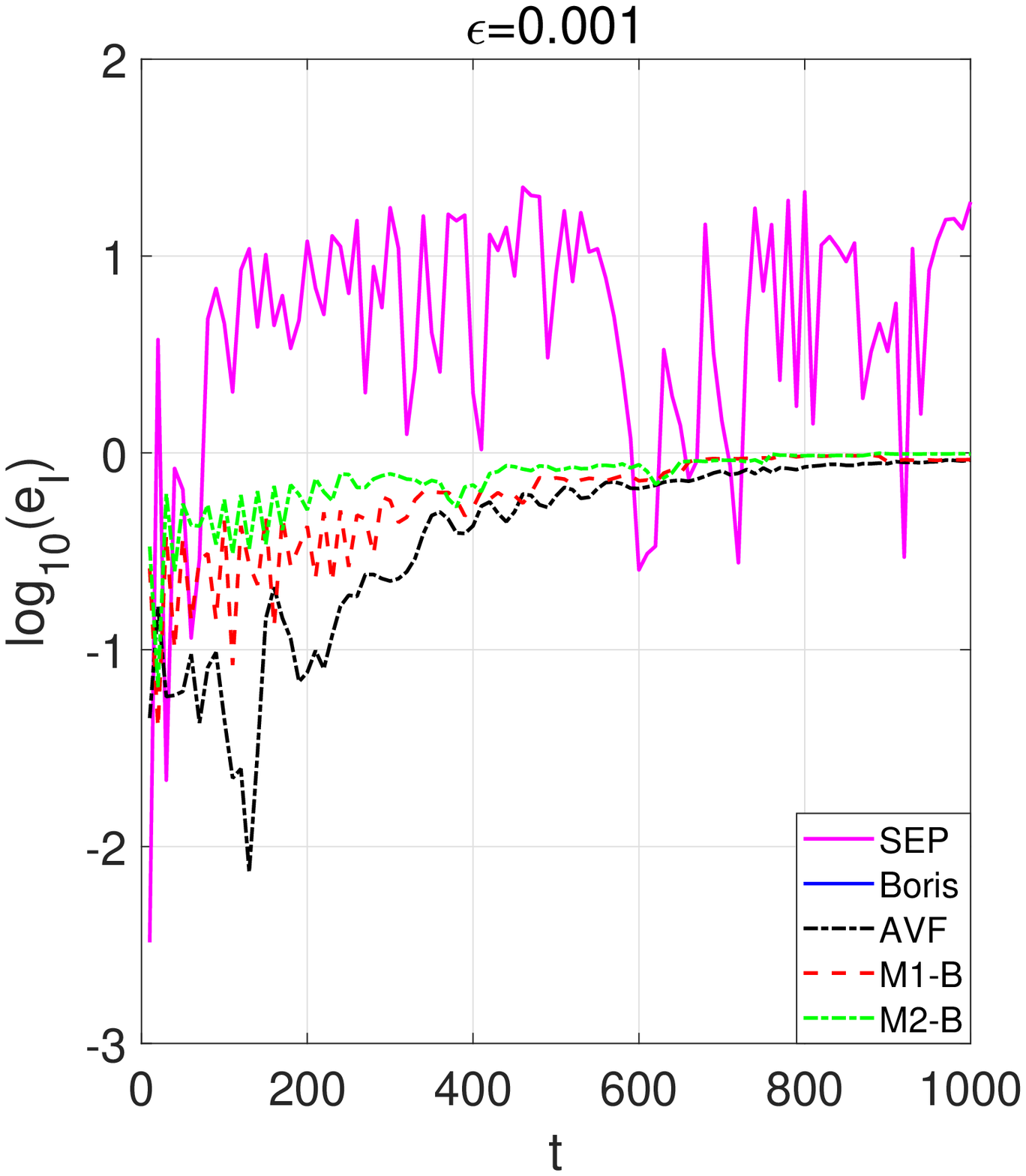} & \includegraphics[width=4.6cm,height=4.6cm]{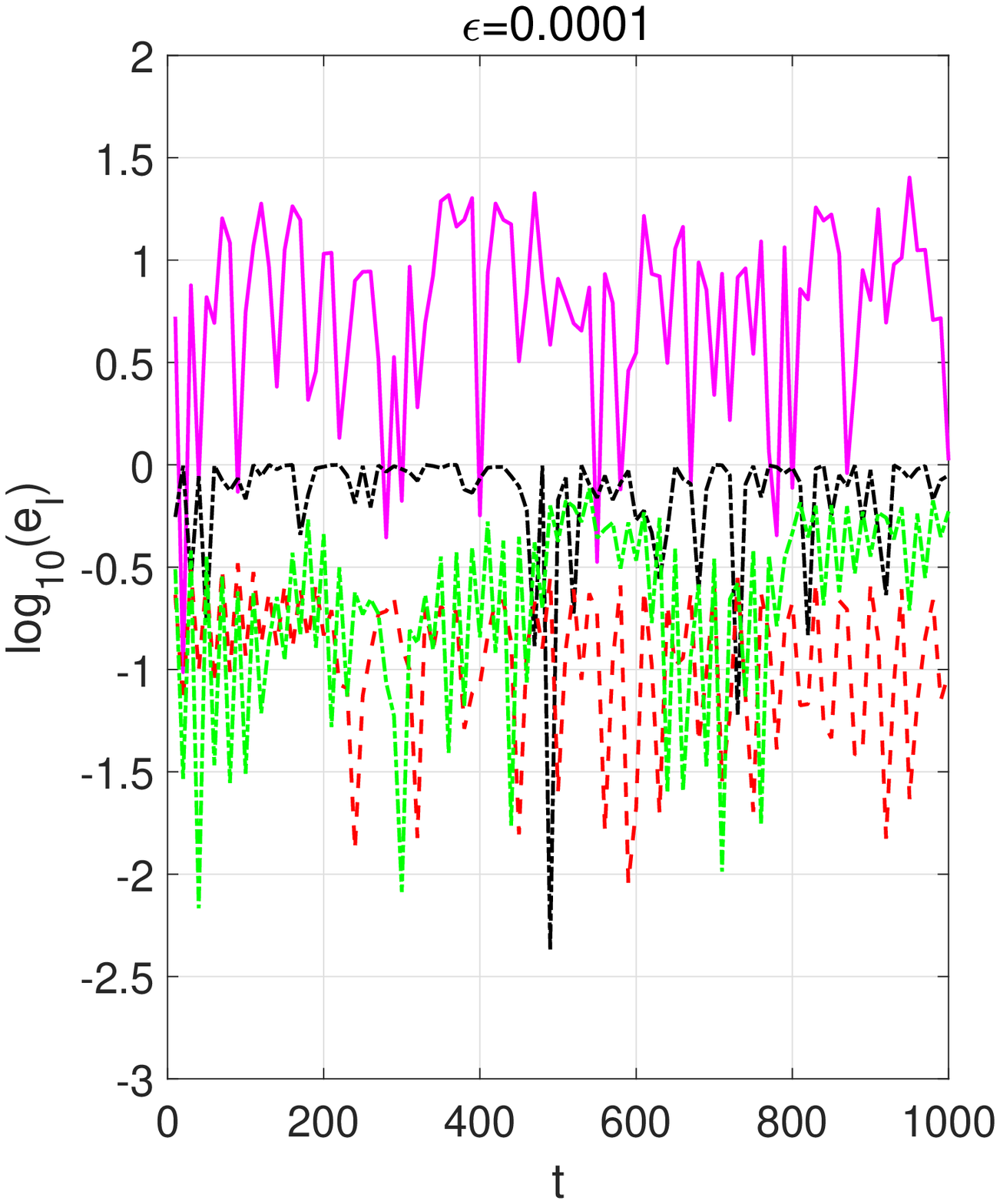} & \includegraphics[width=4.6cm,height=4.6cm]{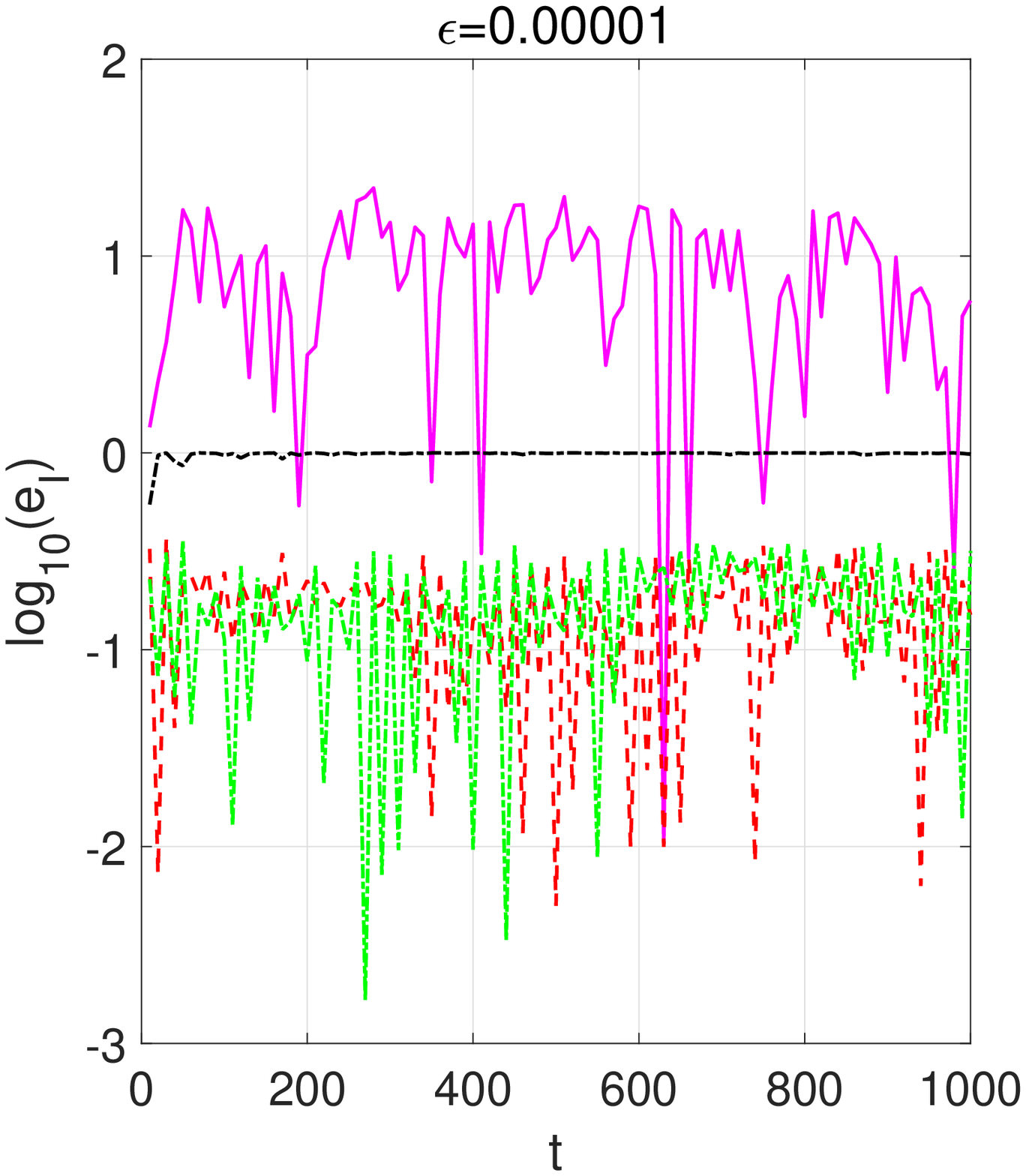}\\
{\small (i)} & {\small (ii)} & {\small (iii)}\\
\end{tabular}
\caption{Problem 4.  Evolution of the magnetic moment error $e_{I}:=\frac{|I(x_{n},v_n)-I(x_0,v_0)|}{|I(x_0,v_0)|}$ as function of time  $t=nh$.}
\label{fig:problem43}
\end{figure}

From the results of these two tests, it can be observed that AVF is no longer energy preserving and other
 numerical phenomena are similar as Problems 1-2.
In conclusion, the M1-C and M2-C respectively show second order and fourth order in the accuracy, conserve the energy with good accuracy and have   long time near conservations in  the momentum and magnetic moment (when $\eps$ is small). The theoretical analysis of the near conservation in  the momentum and magnetic moment will be considered in our next work.

%
%
%

\section{Conclusions}\label{sec:conclu}

In this paper, geometric continuous-stage exponential energy-preserving integrators
  for the charged-particle dynamics (CPD) \eqref{charged-particle sts-cons} were  formulated and developed. The novel integrators were designed for the CPD in a magnetic field from normal to strong regimes, and they work well for these both regimes. We analysed  the energy-preserving conditions, symmetric conditions and order conditions for the novel integrators. Using these results,  two  symmetric continuous-stage exponential energy-preserving  schemes of order up to  four were constructed. The numerical experiments were performed  and  the results showed
 that our novel  methods were more effective than some  existing methods in the literature.

Last but not least, it is noted that the accuracy of the proposed integrators of this paper is not uniform in $\epsilon$ which is different from  the uniformly accurate methods  \cite{VP1,Zhao,VP3,VP8}. The formulation of uniformly accurate methods with exact energy conservation  is interesting but  very challenging.
Very recently, the work \cite{WZ21} succeeded in equipping the favorable uniformly accuracy methods with near conservation laws in long times. We hope to make some  progress on the topic of energy-preserving uniformly accurate methods in our future work.

\end{document}